\documentclass[preprint,12pt]{elsarticle}    % onecolumn (second format)

\usepackage{amsbsy}
\usepackage{graphicx,epsfig}
\usepackage{enumerate}
\usepackage{amssymb}
\usepackage{amsthm}
\usepackage{amsmath}
\usepackage{xcolor}
\usepackage{todonotes}
\usepackage{times}
\usepackage{mathrsfs}
\usepackage[utf8]{inputenc}
\usepackage{bm}
\usepackage{mathtools}
\usepackage{appendix}

\usepackage[linkcolor=blue, urlcolor=blue, citecolor=blue,
colorlinks, bookmarks]{hyperref}

\newtheorem{theorem}{Theorem}[section]
\theoremstyle{definition}
\newtheorem{lemma}[theorem]{Lemma}

\newtheorem{definition}[theorem]{Definition}

\newtheorem{example}[theorem]{Example}
\newtheorem{remark}{Remark}

\newtheorem{note}{Note}

\DeclareMathOperator*{\argmin}{argmin}
\newsavebox\myboxA
\newsavebox\myboxB
\newlength\mylenA

\makeatletter
\newcommand{\thickhline}{%
	\noalign {\ifnum 0=`}\fi \hrule height 1pt
	\futurelet \reserved@a \@xhline
}
\numberwithin{equation}{section}

\journal{Fuzzy Sets and Systems}
\DeclareUnicodeCharacter{0301}{'}
\begin{document}

	\begin{frontmatter}
		
		\title{\textbf{Ekeland's Variational Principle for Interval-valued Functions}}
		\author[iitbhu_math]{Gourav Kumar}
	    \ead{gouravkr.rs.mat17@itbhu.ac.in}
		\author[iitbhu_math]{Debdas Ghosh\corref{cor1}}
		\ead{debdas.mat@iitbhu.ac.in}
		\address[iitbhu_math]{Department of Mathematical Sciences,  Indian Institute of Technology (BHU) Varanasi \\ Uttar Pradesh--221005, India}
		\cortext[cor1]{Corresponding author}

		\begin{abstract}
		In this paper, we attempt to propose Ekeland's variational principle for interval-valued functions (IVFs). To develop the variational principle, we study a concept of sequence of intervals. In the sequel, the idea of $gH$-semicontinuity for IVFs is explored. A necessary and sufficient condition for an IVF to be $gH$-continuous in terms of $gH$-lower and upper semicontinuity is given. Moreover, we prove a  characterization for $gH$-lower semicontinuity by the level sets of the IVF. With the help of this characterization result, we ensure the existence of a minimum for an extended $gH$-lower semicontinuous, level-bounded and proper IVF. To find an approximate minima of a $gH$-lower semicontinuous and $gH$-G\^{a}teaux differentiable IVF, the proposed Ekeland's variational principle is used.
		\end{abstract}
	
		\begin{keyword}
				Interval-valued functions, $gH$-semicontinuity, $gH$-G\^{a}teaux differentiability, Ekeland's variational principle\\
			
			 Mathematics Subject Classification: 26A24 $\cdot$ 90C30 $\cdot$ 65K05
		\end{keyword}
		
	\end{frontmatter}
	
	\section{Introduction}
	In real analysis, we deal with real-valued functions and their calculus. Similarly, interval analysis deals with interval-valued functions (IVFs), where uncertain variables are represented by intervals. The analysis of IVFs enables one to effectively deal with the errors/uncertainties that appear while modeling the practical problems. Development of the theories of IVFs is primarily important for the analysis of fuzzy-number-valued functions since alpha-cuts of a fuzzy number is a compact interval \cite{ghosh2019introduction}. In fact, for a given alpha, the alpha-cut of a fuzzy-number-valued function is an interval-valued function. Thus, once the tools to analyze IVFs are ready, by the decomposition principle \cite{ghosh2019introduction}, one can easily investigate the properties of fuzzy-valued functions. \\

	To identify characteristic of IVFs, calculus plays a significant role. Wu \cite{wu2007karush} proposed the concepts of limit, continuity, and $H$-differentiability for IVFs. The concept of $H$-differentiability uses $H$-difference to find the difference between elements of $I(\mathbb{R})$, and hence it is restrictive \cite{stefanini2009generalized}. To overcome the  shortcomings of $H$-differentiability, Stefanini and Bede \cite{stefanini2009generalized} introduced $gH$-differentiability for IVFs. Thereafter, by using $gH$-differentiability, Chalco-Cano et al. \cite{chalco2013calculus} developed the calculus for IVFs. In the same article \cite{chalco2013calculus}, the fundamental theorem of calculus for IVFs has been presented. With the help of the parametric representation of an IVF, the notions of $gH$-gradient and $gH$-partial derivative of an IVF has been discussed in \cite{ghosh2017newton}. Recently, Ghosh et al. \cite{ghosh2020generalized} introduced the concepts of $gH$-G\^{a}teaux and Fr\'{e}chet derivatives for IVFs with the help of linear IVFs. Further, researchers have also discussed concepts of differential equations with IVFs \cite{ahmad2019sufficiency,chen2004interval,van2015initial,wu2007karush}.  In order to study the interval fractional differential equations, Lupulescu \cite{lupulescu2015fractional} developed the theory of fractional calculus for IVFs. \\

%In the same article \cite{lupulescu2015fractional} the properties of Riemann–Liouville fractional integral and Caputo fractional derivative for IVFs are investigated by using $gH$-difference.

	In developing mathematical theory for optimization with IVFs, apart from calculus of IVFs, an appropriate choice for ordering of intervals is necessary since the set of intervals is not linearly ordered \cite{ghosh2020variable} like the set of real numbers. Hence, the very definition of optimality gets differed than that of conventional one. However, one can use some partial ordering structures on the set of intervals. Some partial orderings of intervals are discussed by Ishibuchi and Tanka in their 1990 paper \cite{ishibuchi1990multiobjective}. By making use of these partial orderings and $H$-differentiability, Wu \cite{wu2007karush} proposed KKT optimality conditions for an IOP. In a set of two papers, Wu \cite{wu2008interval,wu2008wolfe} solved four types of IOPs and presented weak and strong duality theorems for IOPs by using $H$-differentiability. Chalco-Cano et al. \cite{chalco2013optimality} used a more general concept of differentiability ($gH$-differentiability) and provided KKT type optimality conditions for IOPs. Singh et al. \cite{singh2016kkt} investigated a class of interval-valued multiobjective programming problems and proposed the concept of Pareto optimal solutions for this class of optimization problems. Unlike the earlier approaches, in 2017, Osuna-G{\'o}mez et al. \cite{osuna2017new} provided efficiency conditions for an IOP without converting it into a real-valued optimization problem. In 2018, Zhang et al. \cite{zhang2018multi} and Gong et al. \cite{gong2016set} proposed genetic algorithms to solve IOPs.  Ghosh et al. \cite{Ghosh2019extended} reported generalized KKT conditions to obtain the solution of constrained IOPs.
	Many other authors have also proposed optimality conditions and solution concepts for IOP, for instances, see \cite{ahmad2019sufficiency,ghosh2017newton,ghosh2018saddle,wolfe2000interval} and the references therein.\\
	
	%\todo[inline, color=red!40]{In the above paragraph include four/five papers on IOPs from good journals. Also write one/two lines about those papers. }
	
% 	The interval analysis, developed so far, have helped to solve many real-life problems, for instances, planning of municipal solid waste management systems \cite{wu2006interval}, fault detection and isolation of the three-tank system \cite{sainz2002fault}, household load scheduling with interval uncertainties \cite{wang2016interval}, application to electronic wearable system of a smart watch \cite{tang2017dimension}, electric water heater scheduling in uncertain environments \cite{wang2019mpc}, application to the structure design of a satellite \cite{fu2019interval}, etc. In the direction to use interval analysis to solve real-life problems,  Sainz et al. \cite{sainz2008continuous} proposed a new approach to the solution of continuous minimax problems by using modal intervals.

\subsection{Motivation and Work Done}
So far, all the solution concepts in interval analysis to find minima of an IVF are for those IVFs that are $gH$-continuous and $gH$-differentiable. However, while modeling the real-world problems, we may get an objective function that is neither $gH$-differentiable nor $gH$-continuous\footnote{Analytical models of some interesting real-world problems with neither differentiable nor continuous objective functions can be found in Clarke's book \cite{clarke1990optimization}.}. We thus,   in this study, introduce the notions of $gH$-semicontinuity and give results which guarantees the existence of a minima and an approximate minima for an IVF which need not be $gH$-differentiable or $gH$-continuous. \\

	%\todo[inline, color=red!40]{Write here two/three sentences that link nondifferentiability/discontinuty with Ekland's variational principle. Actually, a link is missing between the two paragraphs in motivation. }

%\todo[inline,color=red!40]{@Gourav: Add a reference or an example that supports the last line. You can give an example of a general model problem. Also, just one line motivation is not enough. Write 5-7 lines mentioning the importance of your theories for the case of real-valued functions. The motivation and importance for Ekland's principle that you mentioned is perfectly fine. }

% We know that many real-life problems can be modeled as optimization problems in which the objective function need not be differentiable not even continuous.

For a nonsmooth optimization problem, it is not always easy to find an exact optima \cite{facchinei2007finite}. In such situations, one attempts to find approximate optima.  It is a well-known fact that Ekeland's variational principle \cite{ekeland1974variational} is helpful to give approximate solutions \cite{facchinei2007finite}. Also, it is widely known that in the conventional and vector optimization problems, the concept of weak sharp minima \cite{burke2002weak} plays an important role. It is closely related to sensitive analysis and convergence analysis of optimization problems \cite{burke1993weak,henrion2001subdifferential}. Ekeland's variational principle is a useful tool to show the existence of weak sharp minima for a constrained optimization problem with nonsmooth objective function \cite{facchinei2007finite}.
Moreover, Ekeland's variational principle \cite{ekeland1974variational} is one of the most powerful tools for nonlinear analysis. It has applications in different areas including optimization theory, fixed point theory, and global analysis, for instances, see \cite{borwein1999equivalence,ekeland1979nonconvex,fabian1996smooth,fabian1998nonsmooth,georgiev1988strong,kruger2003frechet,penot1986drop}. Due to all these wide applications of Ekeland's variational principle in different areas, especially in nonsmooth optimization and control theory, we attempt to study this principle for $gH$-lower semicontinuous IVFs in this article. Further, we also give Ekeland's variational principle for $gH$-G\^{a}teaux differentiable IVFs.

%A way to show the existence of weak sharp minima for a constrained optimization problem with nonsmooth objective function is by using Ekeland's variational principle
%\todo[inline,color=red!40]{@Gourav: A link is missing from the previous to the next paragraphs. Write one or two sentences linking $gH$-subdifferentiability to Ekland's variational principle. }	
	
\subsection{Delineation}	
	The proposed study is presented in the following manner. In Section \ref{sec2}, basic terminologies and definitions on intervals and IVFs are provided. In Section \ref{sec3}, we define $gH$-semicontinuity for IVFs and give a characterization for $gH$-continuity of an IVF in terms of $gH$-lower and upper semicontinuity. Also, we give a characterization of $gH$-lower semicontinuity, and using this we prove that an extended $gH$-lower semicontinuous,  level-bounded and proper IVF attains its minimum. Further, a characterization of the set argument minimum of an IVF is given. After that, we present Ekeland's variational principle for IVFs and its application in Section \ref{sec5}. Lastly, the conclusion and future scopes are given in section \ref{sec6}.

\section{Preliminaries and Terminologies}\label{sec2}
In this article, the following notations are used.

\begin{itemize}[{\tiny\raisebox{1ex}{\textbullet}}]
    \item $\mathbb{R}$ denotes the set of real numbers
    \item $\mathbb{R}^+$ denotes the set of nonnegative real numbers
    \item $I(\mathbb{R})$ represents the set of all closed and bounded intervals
    \item  Bold capital letters are used to represent the elements of $I(\mathbb{R})$
    \item $\overline{I(\mathbb{R})}=I(\mathbb{R})\cup\{-\infty,+\infty\}$
     \item $\textbf{0}$ represents the interval $[0,0]$
    \item $\mathcal{X}$ denotes a finite dimensional Banach space
    \item $B_{\delta}(\bar{x})~\text{is an open ball of radius}~\delta~\text{centered at}~\bar{x}$.
\end{itemize}
	
	Consider two intervals ${\textbf A} = [\underline{a}, \overline{a}]$ and $\textbf{B} = \left[\underline{b}, \overline{b}\right]$.
	The \emph{addition} of ${\textbf A}$ and ${\textbf B}$, denoted ${\textbf A} \oplus {\textbf B}$, is defined by
	\[
	{\textbf A} \oplus {\textbf B} =\left[~\underline{a} + \underline{b},~ \overline{a} + \overline{b}~\right].
	\]
	The \emph{addition} of ${\textbf A}$ and a real number $a$, denoted ${\textbf A} \oplus a$, is defined by
	\[\textbf{A}\oplus a= \textbf{A}\oplus[a,a]=\left[~\underline{a} + a,~ \overline{a} + a~\right].\]
	The \emph{subtraction} of ${\textbf B}$ from ${\textbf A}$, denoted ${\textbf A} \ominus {\textbf B}$, is defined by
	\[
	{\textbf A} \ominus {\textbf B} =  \left[~\underline{a} - \overline{b},~ \overline{a} - \underline{b}~\right].
	\]
	The \emph{multiplication} by a real number $\mu$ to  ${\textbf A}$, denoted $\mu \odot  {\textbf A}$ or ${\textbf A} \odot \mu$, is defined by
	\[
	\mu \odot  {\textbf A} =  {\textbf A} \odot  \mu =
	\begin{cases}
	[\mu \underline{a},~\mu \overline{a}], & \text{if $\mu \geq 0$}\\
	[\mu \overline{a},~\mu \underline{a}], & \text{if $\mu < 0.$}
	\end{cases}
	\]
	\begin{definition} (\emph{$gH$-difference of intervals} \cite{stefanini2009generalized}). Let $\textbf{A}$ and $\textbf{B}$ be two elements of $I(\mathbb{R})$. The $gH$-difference between $\textbf{A}$ and $\textbf{B}$ is defined as the interval $\textbf C$ such that
		\begin{equation*}
		\textbf{C}=\textbf{A}\ominus_{gH}\textbf{B}\Longleftrightarrow \begin{cases*}
		\textbf{A}=\textbf{B}\oplus \textbf{C}\\
		\text{or}\\
		\textbf{B}=\textbf{A}\ominus \textbf{C}.
		\end{cases*}
		\end{equation*}
	\end{definition}
	\noindent For $\textbf{A}=[\underline{a},\overline{a}]$ and $\textbf{B}=[\underline{b},\overline{b}]$, $\textbf{A}\ominus_{gH}\textbf{B}$ is given by (see \cite{stefanini2009generalized})
	 \[\textbf{A}\ominus_{gH}\textbf{B}=[\min\{\underline{a}-\underline{b},\overline{a}-\overline{b}\},\max\{\underline{a}-\underline{b},\overline{a}-\overline{b}\}].\]
	Also, if $\textbf{A}=[\underline{a},\overline{a}]$ and $a$ be any real number, then we have
	\[\textbf{A}\ominus_{gH} a=\textbf{A}\ominus_{gH}[a,a]=[\min\{\underline{a}-a,\overline{a}-a\},\max\{\underline{a}-a,\overline{a}-a\}].\]

	%For more details on interval arithmetic see \cite{ghosh2017newton,moore2009introduction,stefanini2009generalized} and their references.
	\begin{definition}\label{g99}(\emph{Dominance of intervals} \cite{wu2008wolfe}). Let $\textbf{A}=[\underline{a},\overline{a}]$ and $\textbf{B}=[\underline{b},\overline{b}]$ be two elements of $I(\mathbb{R})$. Then,
		\begin{enumerate}[(i)]
			\item $\textbf{B}$ is said to be dominated by $\textbf{A}$ if $\underline{a}\leq\underline{b}$ and $\overline{a}\leq\overline{b}$, and then we write $\textbf{A}\preceq \textbf{B}$;
			\item $\textbf{B}$ is said to be strictly dominated by $\textbf{A}$ if $\textbf{A}\preceq \textbf{B}$ and $\textbf{A}\neq\textbf{B}$, and then we write $\textbf{A}\prec\textbf{B}.$ Equivalently, $\textbf{A}\prec \textbf{B}$ if and only if any of the following cases hold:
			\begin{enumerate}[$\bullet$ \textbf{Case} 1.]
			\item $\underline{a}<\underline{b}$ and $\overline{a}\leq \overline{b}$,
			\item $\underline{a}\leq \underline{b}$ and $\overline{a}<\overline{b}$,
			\item $\underline{a}<\underline{b}$ and $\overline{a}< \overline{b}$;
			\end{enumerate}
			\item if neither $\textbf{A}\preceq \textbf{B}$ nor $\textbf{B}\preceq \textbf{A}$, we say that none of $\textbf{A}$ and $\textbf{B}$ dominates the other, or $\textbf{A}$ and $\textbf{B}$ are not comparable. Equivalently, $\textbf{A}$ and $\textbf{B}$ are not comparable if either `$\underline{a}<\underline{b}$ and $\overline{a}>\overline{b}$' or `$\underline{a}>\underline{b}$ and $\overline{a}<\overline{b}$';
			\item $\textbf{B}$ is said to be not dominated by $\textbf{A}$ if either  $\textbf{B}\preceq\textbf{A}$ or $\textbf{A}$ and $\textbf{B}$ are not comparable, and then we write $\textbf{A}\nprec\textbf{B}.$ Similarly, a real number $a$ is said to be not dominated by $\textbf{A}$ if either $[a,a]\preceq\textbf{A}$ or $\textbf{A}$ and $[a,a]$ are not comparable, and then we write $\textbf{A}\nprec a$.
		\end{enumerate}
	\end{definition}
	\begin{remark}
		By Definition \ref{g99}, it is easy to see that for any $\textbf{A},~\textbf{B}\in I(\mathbb{R})$ either $\textbf{A}\prec\textbf{B}$ or $\textbf{A}\nprec\textbf{B}$.
	\end{remark}
	%	\begin{lemma}(See \cite{ghosh2020generalized}). \label{g73} For two elements $\textbf{A}$ and $\textbf{B}$ of $I(\mathbb{R})$,
	%	\begin{enumerate}[(i)]
	%		\item \label{part1} $\textbf{A}\preceq \textbf{B} ~\Longleftrightarrow~ \textbf{A}\ominus_{gH}\textbf{B}\preceq  \textbf{0}$~\text{ and }
	%		\item \label{part2} $\textbf{A}\nprec \textbf{B} ~\Longleftrightarrow~ \textbf{A}\ominus_{gH}\textbf{B}\nprec \textbf{0}.$
	%	\end{enumerate}
	%\end{lemma}
	 %%%%%%%%%%%%%%%%%%%%%%%%%%%%%%%%%%%%%%%%%%%%%%%%%%%%%%%%%%%%%%%%%%%%%%%%%%%%%%%%%%%%%%%%%%%%%%%%%%%%%%%%%%%%%%%%%%%%%%%%%%%%%%%%%%%%%%%%%%%%%%%%%%%%%%%%%%%%%%%%%%%%%%%%%%%%%%%%%%%%%%%%%%%%%%%%%%%%%%%%%%%%%%%%%%%%%%%%%%%%%%%%%%%%%%%%%%%%%%%%%%%%%%%%%%%%%%%%%%%%%%%%%%%%%%%%%%%%%%%%%%%%%%%%%%%%%%%%%%%%%%%%%%%%%%%%%%%%%%%%%%%%%%%%%%%%%%%%%%%%%%
	 %%%%%%%%%%%%%%%%%%%%%%%%%%%%%%%%%%%%%%%%%%%%%%%%%%%%%%%%%%%%%%%%%%%%%%%%%%%%%%%%%%%%%%%%%%%%%%%%%%%%%%%%%%%%%%%%%%%%%%%%%%%%%%%%%%%%%%%%%%%%%%%%%%%%%%%%%%%%%%%%%%%%%%%%%%%%%%%%%%%%%%%%%%%%%%%%%%%%%%%%%%%%%%%%%%%%%%%%%%%%%%%%%%%%%%%%%%%%%%%%%%%%%%%%%%%%%%%%%%%%%%%%%%%%%%%%%%%%%%%%%%%%%%%%%%%%%%%%%%%%%%%%%%%%%%%%%%%%%%%%%%%%%%%%%%%%%%%%%%%%%%
	In the following two lemmas, we give a few inequalities about intervals and their norms. The norm of an interval $\textbf{A} = \left[\underline{a}, \bar{a}\right]$ is defined by (see \cite{moore1966interval})
		\[
		{\lVert \textbf{A} \rVert}_{I(\mathbb{R})} = \max \{|\underline{a}|, |\bar{a}|\}.
		\]
	It is noteworthy that the set $I(\mathbb{R})$ equipped with the norm ${\lVert . \rVert}_{I(\mathbb{R})}$ is a normed quasilinear space with respect to the operations $\oplus, \ominus_{gH}$ and $\odot$ (see \cite{lupulescu2015fractional}).

	%\begin{lemma}\label{g27} (See \cite{ghosh2020generalized}). For all $\textbf{A}, \textbf{B}$, $\textbf{C}\in I(\mathbb{R})$,
	%	\begin{equation*}\label{p1}
	%	{\lVert \textbf{A}\ominus_{gH}\textbf{B} \rVert}_{I(\mathbb{R})}
	%	\leq
	%	{\lVert (\textbf{A} \ominus_{gH} \textbf{C})\oplus(\textbf{C}\ominus_{gH} \textbf{B})
%\rVert}_{I(\mathbb{R})}.
		%\end{equation*}
%	\end{lemma}
%%%%%%%%%%%%%%%%%%%%%%%%%%%%%%%%%%%%%%%%%%%%%%%%%%%%%%%%%%%%%%%%%%%%%%%%%%%%%%%%%%%%%%%%%%%%%%%%%%%%%%%%%%%%%%%%%%%%%%%%%%%%%%%%%%%%%%%%%%%%%%%%%%%%%%%%%%%%%%%%%%%%%%%%%%%%%%%%%%%%%%%%%%%%%%%%%%%%%%%%%%%%%%%%%%%%%%%%%%%%%%%%%%%%%%%%%%%%%%%%%%%%%%%%%%%%%%%%%%%%%%%%%%%%%%%%%%%%%%%%%%%%%%%%%%%%%%%%%%%%%%%%%%%%%%%%%%%%%%%%%%%%%%%%%%%%%%%%%%%%%%
	\begin{lemma}\label{g23} Let  $\textbf{A},~\textbf{B},~\textbf{C}$ and $\textbf{D}$ be elements of $I(\mathbb{R})$. Then,
		\begin{enumerate}[(i)]
			\item\label{1} %Triangle Inequality for elements of $I(\mathbb{R}):$%For two elements $\textbf{A}$ and $\textbf{B}$ of $I(\mathbb{R})$,
			%\begin{eqnarray*}
			$\lVert \textbf{A}\oplus \textbf{B}\rVert_{I(\mathbb{R})}\leq \lVert\textbf{A}\rVert_{I(\mathbb{R})}+\lVert\textbf{B}\rVert_{I(\mathbb{R})}$ \big(triangle inequality for the elements of $I(\mathbb{R})$\big),
			%\end{eqnarray*}
			\item \label{2}	%Let $\textbf{A},~\textbf{B},~\textbf{C}$ and $\textbf{D}$ be elements of $I(\mathbb{R})$ such that
			if $\textbf{A}\preceq \textbf{C}$ and $\textbf{B}\preceq \textbf{D}$, then $\textbf{A}\oplus \textbf{B}\preceq \textbf{C}\oplus \textbf{D}$,
		%	\item \label{3}\text{if}~$\textbf{0}\preceq \textbf{A}~\text{and}~\textbf{A}\prec\textbf{B},~\text{then}~\lVert\textbf{A}\rVert_{I(\mathbb{R})}\leq\lVert\textbf{B}\rVert_{I(\mathbb{R})}$, and
		%	\item\label{4} $\textbf{A}\nprec\textbf{B}\oplus[\epsilon, \epsilon] \implies\textbf{A}\npreceq\textbf{B}$, where $\epsilon>0$ is any real number.
		\end{enumerate}
	\end{lemma}
	\begin{proof}
		See \ref{g24}.
	\end{proof}
	 %%%%%%%%%%%%%%%%%%%%%%%%%%%%%%%%%%%%%%%%%%%%%%%%%%%%%%%%%%%%%%%%%%%%%%%%%%%%%%%%%%%%%%%%%%%%%%%%%%%%%%%%%%%%%%%%%%%%%%%%%%%%%%%%%%%%%%%%%%%%%%%%%%%%%%%%%%%%%%%%%%%%%%%%%%%%%%%%%%%%%%%%%%%%%%%%%%%%%%%%%%%%%%%%%%%%%%%%%%%%%%%%%%%%%%%%%%%%%%%%%%%%%%%%%%%%%%%%%%%%%%%%%%%%%%%%%%%%%%%%%%%%%%%%%%%%%%%%%%%%%%%%%%%%%%%%%%%%%%%%%%%%%%%%%%%%%%%%%%%%%%
	\begin{lemma}\label{g28}(\emph{Properties of the elements of $I(\mathbb{R})$ under $gH$-difference}). For all elements $\textbf{A},~\textbf{B},~\textbf{C},~\textbf{D}\in I(\mathbb{R})$ and $\epsilon>0$, we have
		\begin{enumerate}[(i)]
		%	\item\label{5} $\lVert(\textbf{A}\oplus \textbf{B})\ominus_{gH}(\textbf{C}\oplus \textbf{D})\rVert_{I(\mathbb{R})}\leq\lVert\textbf{A}\ominus_{gH}\textbf{C}\rVert_{I(\mathbb{R})}+\lVert\textbf{B}\ominus_{gH}\textbf{D}\rVert_{I(\mathbb{R})}$,
			\item\label{6} $\lVert \textbf{A}\ominus_{gH}\textbf{B}\rVert_{I(\mathbb{R})}<\epsilon\iff\textbf{B}\ominus_{gH}[\epsilon,\epsilon]\prec\textbf{A}\prec\textbf{B}\oplus[\epsilon,\epsilon]$,
		%	\item\label{7} $\textbf{A}\prec\textbf{B}\oplus[\epsilon,\epsilon]\implies\textbf{A}\ominus_{gH}\textbf{B}\prec[\epsilon,\epsilon]$,
			\item\label{8} $\textbf{A}\ominus_{gH}[\epsilon,\epsilon]\nprec\textbf{B}\implies\textbf{A}\npreceq\textbf{B}$,
			%\item\label{11} $\textbf{0}\preceq \textbf{A}\ominus_{gH}\textbf{B}\implies\textbf{B}\preceq\textbf{A},$ and
		%	\item\label{l} if $\textbf{0}\nprec\textbf{A}$ and $\textbf{0}\nprec\textbf{B}$, then $\textbf{0}\nprec\textbf{A}\oplus\textbf{B}$.
		\end{enumerate}
	\end{lemma}
	\begin{proof}
		See \ref{g35}.
	\end{proof}
	 %%%%%%%%%%%%%%%%%%%%%%%%%%%%%%%%%%%%%%%%%%%%%%%%%%%%%%%%%%%%%%%%%%%%%%%%%%%%%%%%%%%%%%%%%%%%%%%%%%%%%%%%%%%%%%%%%%%%%%%%%%%%%%%%%%%%%%%%%%%%%%%%%%%%%%%%%%%%%%%%%%%%%%%%%%%%%%%%%%%%%%%%%%%%%%%%%%%%%%%%%%%%%%%%%%%%%%%%%%%%%%%%%%%%%%%%%%%%%%%%%%%%%%%%%%%%%%%%%%%%%%%%%%%%%%%%%%%%%%%%%%%%%%%%%%%%%%%%%%%%%%%%%%%%%%%%%%%%%%%%%%%%%%%%%%%%%%%%%%%%%%
	 %%%%%%%%%%%%%%%%%%%%%%%%%%%%%%%%%%%%%%%%%%%%%%%%%%%%%%%%%%%%%%%%%%%%%%%%%%%%%%%%%%%%%%%%%%%%%%%%%%%%%%%%%%%%%%%%%%%%%%%%%%%%%%%%%%%%%%%%%%%%%%%%%%%%%%%%%%%%%%%%%%%%%%%%%%%%%%%%%%%%%%%%%%%%%%%%%%%%%%%%%%%%%%%%%%%%%%%%%%%%%%%%%%%%%%%%%%%%%%%%%%%%%%%%%%%%%%%%%%%%%%%%%%%%%%%%%%%%%%%%%%%%%%%%%%%%%%%%%%%%%%%%%%%%%%%%%%%%%%%%%%%%%%%%%%%%%%%%%%%%%%
	\begin{definition}\label{g11}(\emph{Infimum of a subset of $\overline{I(\mathbb{R})}$}). Let $\textbf{S}\subseteq \overline{I(\mathbb{R})}$. An interval $\mathbf{\bar{A}}\in I(\mathbb{R})$ is said to be a lower bound of $\textbf{S}$ if $\mathbf{\bar{A}}\preceq \textbf{B}$ for all $\textbf{B}$ in $\textbf{S}.$ A lower bound $\mathbf{\bar{A}}$ of $\textbf{S}$ is called an infimum of $\textbf{S}$ if for all lower bounds $\textbf{C}$ of $\textbf{S}$ in $I(\mathbb{R}),~\textbf{C}\preceq \mathbf{\bar{A}}.$ We denote infimum of $\textbf{S}$ by $\inf\textbf{S}.$
	\end{definition}
	 %%%%%%%%%%%%%%%%%%%%%%%%%%%%%%%%%%%%%%%%%%%%%%%%%%%%%%%%%%%%%%%%%%%%%%%%%%%%%%%%%%%%%%%%%%%%%%%%%%%%%%%%%%%%%%%%%%%%%%%%%%%%%%%%%%%%%%%%%%%%%%%%%%%%%%%%%%%%%%%%%%%%%%%%%%%%%%%%%%%%%%%%%%%%%%%%%%%%%%%%%%%%%%%%%%%%%%%%%%%%%%%%%%%%%%%%%%%%%%%%%%%%%%%%%%%%%%%%%%%%%%%%%%%%%%%%%%%%%%%%%%%%%%%%%%%%%%%%%%%%%%%%%%%%%%%%%%%%%%%%%%%%%%%%%%%%%%%%%%%%%%
	\begin{example}
		Let $\textbf{S} = \big\{\big[\frac{1}{n},1\big]:n\in \mathbb{N}\big\}$. The set of lower bounds of $\textbf{S}$ is
		\begin{equation*}
		\{[\alpha,\beta]: -\infty<\alpha\leq0 ~\text{and}~ -\infty<\beta\leq1\}.
		\end{equation*}
		Therefore, the infimum of $\textbf{S}$ is $[0,1]$ because $[\alpha,\beta]\preceq [0,1]$ for all $-\infty<\alpha\leq 0 ~\text{and}~ -\infty<\beta\leq1$.
	\end{example}
	 %%%%%%%%%%%%%%%%%%%%%%%%%%%%%%%%%%%%%%%%%%%%%%%%%%%%%%%%%%%%%%%%%%%%%%%%%%%%%%%%%%%%%%%%%%%%%%%%%%%%%%%%%%%%%%%%%%%%%%%%%%%%%%%%%%%%%%%%%%%%%%%%%%%%%%%%%%%%%%%%%%%%%%%%%%%%%%%%%%%%%%%%%%%%%%%%%%%%%%%%%%%%%%%%%%%%%%%%%%%%%%%%%%%%%%%%%%%%%%%%%%%%%%%%%%%%%%%%%%%%%%%%%%%%%%%%%%%%%%%%%%%%%%%%%%%%%%%%%%%%%%%%%%%%%%%%%%%%%%%%%%%%%%%%%%%%%%%%%%%%%%
	\begin{definition}\label{12}(\emph{Supremum of a subset of $\overline{I(\mathbb{R})}$}). Let $\textbf{S}\subseteq \overline{I(\mathbb{R})}$. An interval $\mathbf{\bar{A}}\in I(\mathbb{R})$ is said to be an upper bound of $\textbf{S}$ if $\textbf{B}\preceq \mathbf{\bar{A}}$ for all $\textbf{B}$ in $\textbf{S}$. An upper bound $\mathbf{\bar{A}}$ of $\textbf{S}$ is called a supremum of $\textbf{S}$ if for all upper bounds $\textbf{C}$ of $\textbf{S}$ in $I(\mathbb{R}),~ \mathbf{\bar{A}}\preceq \textbf{C}$. We denote supremum of $\textbf{S}$ by $\sup\textbf{S}$.
	\end{definition}
	 %%%%%%%%%%%%%%%%%%%%%%%%%%%%%%%%%%%%%%%%%%%%%%%%%%%%%%%%%%%%%%%%%%%%%%%%%%%%%%%%%%%%%%%%%%%%%%%%%%%%%%%%%%%%%%%%%%%%%%%%%%%%%%%%%%%%%%%%%%%%%%%%%%%%%%%%%%%%%%%%%%%%%%%%%%%%%%%%%%%%%%%%%%%%%%%%%%%%%%%%%%%%%%%%%%%%%%%%%%%%%%%%%%%%%%%%%%%%%%%%%%%%%%%%%%%%%%%%%%%%%%%%%%%%%%%%%%%%%%%%%%%%%%%%%%%%%%%%%%%%%%%%%%%%%%%%%%%%%%%%%%%%%%%%%%%%%%%%%%%%%%
	\begin{example}
		Let $\textbf{S} = \big\{\big[\frac{1}{n^2}+1,3\big]:n\in \mathbb{N}\big\}$. The set of upper bounds of $\textbf{S}$ is
		\begin{equation*}
		\{[\alpha,\beta]:  2\leq \alpha<+\infty ~\text{and}~ 3\leq \beta<+\infty\}.
		\end{equation*}
		Therefore, the supremum of $\textbf{S}$ is $ [2,3]$ because $[2,3]\preceq [\alpha,\beta]$ for all $2\leq \alpha<+\infty ~\text{and}~ 3\leq \beta<+\infty.$
	\end{example}
	 %%%%%%%%%%%%%%%%%%%%%%%%%%%%%%%%%%%%%%%%%%%%%%%%%%%%%%%%%%%%%%%%%%%%%%%%%%%%%%%%%%%%%%%%%%%%%%%%%%%%%%%%%%%%%%%%%%%%%%%%%%%%%%%%%%%%%%%%%%%%%%%%%%%%%%%%%%%%%%%%%%%%%%%%%%%%%%%%%%%%%%%%%%%%%%%%%%%%%%%%%%%%%%%%%%%%%%%%%%%%%%%%%%%%%%%%%%%%%%%%%%%%%%%%%%%%%%%%%%%%%%%%%%%%%%%%%%%%%%%%%%%%%%%%%%%%%%%%%%%%%%%%%%%%%%%%%%%%%%%%%%%%%%%%%%%%%%%%%%%%%%
	\begin{remark}\label{g54}
		Let $\textbf{S}=\left\{[a_\alpha,b_\alpha]\in  \overline{I(\mathbb{R})}: \alpha\in\Lambda~ \text{and}~ \Lambda~\text{being an index set}~\right\}$. Then, by Definition \ref{g11} and \ref{12}, it follows that $\inf\textbf{S}=\left[\inf\limits_{\alpha\in\Lambda}a_\alpha,~\inf\limits_{\alpha\in\Lambda}b_\alpha\right]$ and $\sup\textbf{S}=\left[\sup\limits_{\alpha\in\Lambda}a_\alpha,~\sup\limits_{\alpha\in\Lambda}b_\alpha\right].$ It is evident that if $\inf \textbf{S}$ and $\sup \textbf{S}$ exist for an $\textbf{S}$, then they are unique.
	\end{remark}
	 %%%%%%%%%%%%%%%%%%%%%%%%%%%%%%%%%%%%%%%%%%%%%%%%%%%%%%%%%%%%%%%%%%%%%%%%%%%%%%%%%%%%%%%%%%%%%%%%%%%%%%%%%%%%%%%%%%%%%%%%%%%%%%%%%%%%%%%%%%%%%%%%%%%%%%%%%%%%%%%%%%%%%%%%%%%%%%%%%%%%%%%%%%%%%%%%%%%%%%%%%%%%%%%%%%%%%%%%%%%%%%%%%%%%%%%%%%%%%%%%%%%%%%%%%%%%%%%%%%%%%%%%%%%%%%%%%%%%%%%%%%%%%%%%%%%%%%%%%%%%%%%%%%%%%%%%%%%%%%%%%%%%%%%%%%%%%%%%%%%%%%
	\begin{note}\label{p4}
Infimum and supremum of a subset of $\overline{I(\mathbb{R})}$ may not exist. For instance, consider  $\textbf{S}=\{[-2,-1],[-3,-1],[-4,-1],\cdots\}$. Here, $\textbf{S}$ has no lower bound in $I(\mathbb{R})$ as $\{-2,-3,-4,\cdots\}$ has no lower bound in $\mathbb{R}.$ Therefore, infimum of $\textbf{S}$ does not  exist in $I(\mathbb{R})$.
	\end{note}
	 %%%%%%%%%%%%%%%%%%%%%%%%%%%%%%%%%%%%%%%%%%%%%%%%%%%%%%%%%%%%%%%%%%%%%%%%%%%%%%%%%%%%%%%%%%%%%%%%%%%%%%%%%%%%%%%%%%%%%%%%%%%%%%%%%%%%%%%%%%%%%%%%%%%%%%%%%%%%%%%%%%%%%%%%%%%%%%%%%%%%%%%%%%%%%%%%%%%%%%%%%%%%%%%%%%%%%%%%%%%%%%%%%%%%%%%%%%%%%%%%%%%%%%%%%%%%%%%%%%%%%%%%%%%%%%%%%%%%%%%%%%%%%%%%%%%%%%%%%%%%%%%%%%%%%%%%%%%%%%%%%%%%%%%%%%%%%%%%%%%%%%
%	\begin{example}\label{g47}
%		Let $\textbf{S}=\{[-2,-1],[-3,-1],[-4,-1],\cdots\}$. Here, $\textbf{S}$ has no lower bound in $I(\mathbb{R})$ as $\{-2,-3,-4,\cdots\}$ has no lower bound in $\mathbb{R}.$
	%	Therefore, infimum of $\textbf{S}$ does not  exist in $I(\mathbb{R})$.
%	\end{example}
	 %%%%%%%%%%%%%%%%%%%%%%%%%%%%%%%%%%%%%%%%%%%%%%%%%%%%%%%%%%%%%%%%%%%%%%%%%%%%%%%%%%%%%%%%%%%%%%%%%%%%%%%%%%%%%%%%%%%%%%%%%%%%%%%%%%%%%%%%%%%%%%%%%%%%%%%%%%%%%%%%%%%%%%%%%%%%%%%%%%%%%%%%%%%%%%%%%%%%%%%%%%%%%%%%%%%%%%%%%%%%%%%%%%%%%%%%%%%%%%%%%%%%%%%%%%%%%%%%%%%%%%%%%%%%%%%%%%%%%%%%%%%%%%%%%%%%%%%%%%%%%%%%%%%%%%%%%%%%%%%%%%%%%%%%%%%%%%%%%%%%%%
	\begin{remark}
		\begin{enumerate}[(i)]
			\item It is noteworthy that infimum and supremum of a subset of $\overline{I(\mathbb{R})}$ always exist in $\overline{I(\mathbb{R})}$. For instance, consider $\textbf{S}$ as in Note \ref{p4}.
			%i.e.,\\ $\textbf{S}=\{[-2,-1],[-3,-1],[-4,-1],\cdots\}$. %Here, $\textbf{S}$ has no lower bound in $I(\mathbb{R})$ as $\{-2,-3,-4,\cdots\}$ has no lower bound in $\mathbb{R}.$\\
			Here, infimum of $\textbf{S}$ does not  exist in $I(\mathbb{R})$ but exists in $\overline{I(\mathbb{R})}$. Note that infimum of $\textbf{S}$ is $-\infty.$
			\item Infimum and supremum of a finite subset $S$ of real numbers always belong to the set $S$ but this is not true for a finite subset of $I(\mathbb{R}).$ For instance, consider $\textbf{S}=\{[-2,4],[-1,3]\}.$ Then, $\inf \textbf{S}=[-2,3]$ and $\sup\textbf{S}=[-1,4].$
		\end{enumerate}
	\end{remark}
	 %%%%%%%%%%%%%%%%%%%%%%%%%%%%%%%%%%%%%%%%%%%%%%%%%%%%%%%%%%%%%%%%%%%%%%%%%%%%%%%%%%%%%%%%%%%%%%%%%%%%%%%%%%%%%%%%%%%%%%%%%%%%%%%%%%%%%%%%%%%%%%%%%%%%%%%%%%%%%%%%%%%%%%%%%%%%%%%%%%%%%%%%%%%%%%%%%%%%%%%%%%%%%%%%%%%%%%%%%%%%%%%%%%%%%%%%%%%%%%%%%%%%%%%%%%%%%%%%%%%%%%%%%%%%%%%%%%%%%%%%%%%%%%%%%%%%%%%%%%%%%%%%%%%%%%%%%%%%%%%%%%%%%%%%%%%%%%%%%%%%%%
	\begin{definition}\label{g15}(\emph{Infimum of an IVF}). Let $\mathcal{S}$ be a nonempty subset of $\mathcal{X}$ and $\textbf{F}: \mathcal{S}\rightarrow \overline{I(\mathbb{R})}$ be an extended IVF. Then infimum of $\textbf{F}$, denoted as $\inf\limits_{x\in \mathcal{S}} \textbf{F}(x)$ or $\inf\limits_{\mathcal{S}} \textbf{F}$, is equal to the infimum of the range set of $\textbf{F}$, i.e.,
		\begin{equation*}
		{\inf\limits_{\mathcal{S}} \textbf{F}=\inf\{\textbf{F}(x):x\in\mathcal{S}\}.}
		\end{equation*}
		Similarly, the supremum of an IVF is defined by
		\begin{equation*}
		\sup\limits_{\mathcal{S}} \textbf{F}=\sup\{\textbf{F}(x):x\in\mathcal{S}\}.
		\end{equation*}
	\end{definition}
		\begin{definition}(\emph{Sequence in ${I(\mathbb{R})}$}). An IVF $\textbf{F}:\mathbb{N}\rightarrow {I(\mathbb{R})}$ is called a sequence in ${I(\mathbb{R})}$.
	%or sequence of elements of $\overline{I(\mathbb{R})}$.
	The image of $n$th element, $\textbf{F}(n)$, is said to be the $n$th element of the sequence $\textbf{F}.$
		We denote a sequence $\textbf{F}$ by $\{\textbf{F}(n)\}$.
	\end{definition}
	 %%%%%%%%%%%%%%%%%%%%%%%%%%%%%%%%%%%%%%%%%%%%%%%%%%%%%%%%%%%%%%%%%%%%%%%%%%%%%%%%%%%%%%%%%%%%%%%%%%%%%%%%%%%%%%%%%%%%%%%%%%%%%%%%%%%%%%%%%%%%%%%%%%%%%%%%%%%%%%%
	\begin{example}
		\begin{enumerate}[(i)]
			\item  $\textbf{F}:\mathbb{N}\rightarrow {I(\mathbb{R})}$ that is defined by  $\textbf{F}(n)=[n,n+1]$ is a sequence.
			\item $\textbf{F}:\mathbb{N}\rightarrow {I(\mathbb{R})}$ that is defined by $\textbf{F}(n)=\big[\frac{n}{4},\frac{n}{2}\big]$ is also a sequence.
		\end{enumerate}
	\end{example}
	 %%%%%%%%%%%%%%%%%%%%%%%%%%%%%%%%%%%%%%%%%%%%%%%%%%%%%%%%%%%%%%%%%%%%%%%%%%%%%%%%%%%%%%%%%%%%%%%%%%%%%%%%%%%%%%%%%%%%%%%%%%%%%%%%%%%%%%%%%%%%%%%%%%%%%%%%%%%%%%%
	\begin{definition}\label{g22}(\emph{Convergence of a sequence in ${I(\mathbb{R})}$}).
		\begin{enumerate}
			\item A sequence $\{\textbf{F}(n)\}$ is said to converge to $\textbf{L}\in I(\mathbb{R})$ if for each $\epsilon>0$, there exists an integer $m>0$ such that
			\begin{equation*}
			\lVert \textbf{F}(n)\ominus_{gH}\textbf{L}\rVert_{I(\mathbb{R})}<\epsilon~\text{for all}~ n\geq m.
			\end{equation*}
		 The interval $\textbf{L}$ is called limit of the sequence $\{\textbf{F}(n)\}$ and is presented by $\lim\limits_{n\rightarrow +\infty}\textbf{F}(n)=\textbf{L}$ or $\textbf{F}(n)\rightarrow \textbf{L}$.
			\item  We say the limit of a sequence $\{\textbf{F}(n)\}$ is $+\infty$ if for every real number $a>0$, there exists an integer $m>0$ such that
			\begin{equation*}
			[a,a]\prec\textbf{F}(n)~\text{for all}~ n\geq m.
			\end{equation*}
			\item We say the limit of a sequence $\{\textbf{F}(n)\}$ is $-\infty$ if for every real number $a>0$, there exists an integer $m>0$ such that
			\begin{equation*}
			\textbf{F}(n)\prec [-a,-a]~\text{for all}~ n\geq m.
			\end{equation*}
		\end{enumerate}
	\end{definition}
	 %%%%%%%%%%%%%%%%%%%%%%%%%%%%%%%%%%%%%%%%%%%%%%%%%%%%%%%%%%%%%%%%%%%%%%%%%%%%%%%%%%%%%%%%%%%%%%%%%%%%%%%%%%%%%%%%%%%%%%%%%%%%%%%%%%%%%%%%%%%%%%%%%%%%%%%%%%%%%%%
	\begin{example}
		Consider the sequence $\textbf{F}(n)=\left[\frac{1}{n},1\right],~n\in\mathbb{N}$, in ${I(\mathbb{R})}$.\\
		Let $\epsilon>0$ be given. Note that
		\begin{equation*}
		\lVert \textbf{F}(n)\ominus_{gH}[0,1]\rVert_{I(\mathbb{R})}=\left\lVert \left[\frac{1}{n},1\right]\ominus_{gH}[0,1]\right\rVert_{I(\mathbb{R})}=\left\lVert\left[0,\frac{1}{n}\right]\right\rVert_{I(\mathbb{R})}=\frac{1}{n}<\epsilon~\text{whenever}~ n>\frac{1}{\epsilon}.
		\end{equation*}
		So, by taking $m=\lfloor\frac{1}{\epsilon}\rfloor+1$, where $\lfloor \cdot \rfloor$ is the floor function, we have\\
		\begin{equation*}
		\lVert \textbf{F}(n)\ominus_{gH}[0,1]\rVert_{I(\mathbb{R})}<\epsilon~\text{for all}~ n\geq m.
		\end{equation*}
		Thus, $\lim\limits_{n\rightarrow +\infty}\textbf{F}(n)=\textbf{L}=[0,1].$
	\end{example}
	 %%%%%%%%%%%%%%%%%%%%%%%%%%%%%%%%%%%%%%%%%%%%%%%%%%%%%%%%%%%%%%%%%%%%%%%%%%%%%%%%%%%%%%%%%%%%%%%%%%%%%%%%%%%%%%%%%%%%%%%%%%%%%%%%%%%%%%%%%%%%%%%%%%%%%%%%%%%%%%%
%	\todo[inline, color=red!40]{Newly added note}
	\begin{note}
	Let $\left\{\textbf{F}(n)\right\}$ be a sequence in $I(\mathbb{R})$ with $\textbf{F}(n)=\left[\underline{f}(n),\overline{f}(n)\right],$ where $\left\{\underline{f}(n)\right\}$ and $\left\{\overline{f}(n)\right\}$ be two convergent sequences in $\mathbb{R}.$ Then, $\left\{\textbf{F}(n)\right\}$ is convergent and $$\lim\limits_{n\rightarrow +\infty}\textbf{F}(n)=\left[\lim\limits_{n\rightarrow +\infty}\underline{f}(n),\lim\limits_{n\rightarrow +\infty}\overline{f}(n)\right].$$
	The reason is as follows. \\
	
	Suppose $\underline{f}(n)$ and $\overline{f}(n)$ are convergent sequences with limits $l_1$ and $l_2$, respectively. Then, for each $\epsilon>0$, there exist positive integers $m_1$ and $m_2$ such that
	\begin{eqnarray*}
	&&\left\lvert \underline{f}(n)-l_1\right\rvert<\epsilon~\text{for all}~n\geq m_1,~\text{and}~\left\lvert\overline{f}(n)-l_2\right\rvert<\epsilon~\text{for all}~n\geq m_2\\&\iff& \max\left\{\left\lvert \underline{f}(n)-l_1\right\rvert, \left\lvert\overline{f}(n)-l_2\right\rvert\right\}<\epsilon~\text{for all}~n\geq m,~\text{where}~m=\max\{m_1,m_2\}\\&\iff&\left\lVert \left[\underline{f}(n),\overline{f}(n)\right]\ominus_{gH}[l_1,l_2]\right\rVert_{I(\mathbb{R})}<\epsilon~\text{for all}~n\geq m\\&\text{i.e.,}&\left\lVert \textbf{F}(n)\ominus_{gH}[l_1,l_2]\right\rVert_{I(\mathbb{R})}<\epsilon~\text{for all}~n\geq m.
	\end{eqnarray*}
	Thus, $$\lim\limits_{n\rightarrow +\infty}\textbf{F}(n)=[l_1,l_2]=\left[\lim\limits_{n\rightarrow +\infty}\underline{f}(n),\lim\limits_{n\rightarrow +\infty}\overline{f}(n)\right].$$
	\end{note}
		 %%%%%%%%%%%%%%%%%%%%%%%%%%%%%%%%%%%%%%%%%%%%%%%%%%%%%%%%%%%%%%%%%%%%%%%%%%%%%%%%%%%%%%%%%%%%%%%%%%%%%%%%%%%%%%%%%%%%%%%%%%%%%%%%%%%%%%%%%%%%%%%%%%%%%%%%%%%%%%%
	 %%%%%%%%%%%%%%%%%%%%%%%%%%%%%%%%%%%%%%%%%%%%%%%%%%%%%%%%%%%%%%%%%%%%%%%%%%%%%%%%%%%%%%%%%%%%%%%%%%%%%%%%%%%%%%%%%%%%%%%%%%%%%%%%%%%%%%%%%%%%%%%%%%%%%%%%%%%%%%%
	\begin{definition}\label{g42}(\emph{Bounded sequence in ${I(\mathbb{R})}$}). A sequence $\{\textbf{F}(n)\}$ is said to be bounded above if there exists an interval $\mathbf{K}_{1}\in I(\mathbb{R})$ such that
		\begin{equation*}
		\textbf{F}(n)\preceq \mathbf{K}_{1}~\text{for all}~ n\in\mathbb{N}.
		\end{equation*}
		A sequence $\{\textbf{F}(n)\}$ is said to be bounded below if there exists an interval $\mathbf{K}_{2}\in I(\mathbb{R})$ such that
		\begin{equation*}
		\mathbf{K}_{2}\preceq\textbf{F}(n)~\text{for all}~ n\in\mathbb{N}.
		\end{equation*}
		A sequence $\{\textbf{F}(n)\}$ is said to be bounded if it is both bounded above and below.
	\end{definition}
	 %%%%%%%%%%%%%%%%%%%%%%%%%%%%%%%%%%%%%%%%%%%%%%%%%%%%%%%%%%%%%%%%%%%%%%%%%%%%%%%%%%%%%%%%%%%%%%%%%%%%%%%%%%%%%%%%%%%%%%%%%%%%%%%%%%%%%%%%%%%%%%%%%%%%%%%%%%%%%%%
	 %%%%%%%%%%%%%%%%%%%%%%%%%%%%%%%%%%%%%%%%%%%%%%%%%%%%%%%%%%%%%%%%%%%%%%%%%%%%%%%%%%%%%%%%%%%%%%%%%%%%%%%%%%%%%%%%%%%%%%%%%%%%%%%%%%%%%%%%%%%%%%%%%%%%%%%%%%%%%%%
	%\todo[inline, color=red!40]{Newly added concept}
	\begin{definition}
	A sequence $\{\textbf{F}(n)\}$ is said to be monotonic increasing sequence if $\textbf{F}(n)\preceq\textbf{F}(n+1)$ for all $n\in\mathbb{N}.$
	\end{definition}
	 %%%%%%%%%%%%%%%%%%%%%%%%%%%%%%%%%%%%%%%%%%%%%%%%%%%%%%%%%%%%%%%%%%%%%%%%%%%%%%%%%%%%%%%%%%%%%%%%%%%%%%%%%%%%%%%%%%%%%%%%%%%%%%%%%%%%%%%%%%%%%%%%%%%%%%%%%%%%%%%
	\begin{lemma}\label{p}
	A bounded above monotonic increasing sequence of intervals is convergent and converges to its supremum.
	\end{lemma}
	\begin{proof}
	 Let $\{\textbf{F}(n)\}$ be a bounded above monotonic increasing sequence and $\textbf{M}$ be its supremum.\\
	
	 Then, by Definition \ref{12}, we have
	 \begin{enumerate}[(i)]
	     \item  $\textbf{F}(n)\preceq\textbf{M}$ for all $n\in\mathbb{N}$ and
	     \item for a given $\epsilon>0$, there exists an integer $m>0$ such that $\textbf{M}\ominus_{gH}[\epsilon,\epsilon]\prec\textbf{F}(m).$
	 \end{enumerate}
	 Since $\{\textbf{F}(n)\}$ is a monotonic increasing sequence,
	 \[\textbf{M}\ominus_{gH}[\epsilon,\epsilon]\prec\textbf{F}(m)\preceq\textbf{F}(m+1)\preceq\textbf{F}(m+2)\preceq\cdots\preceq\textbf{M}.\]
	 That is, $\textbf{M}\ominus_{gH}[\epsilon,\epsilon]\prec\textbf{F}(m)\prec\textbf{M}\oplus[\epsilon,\epsilon]$ for all $n\geq m$. Thus, the sequence $\{\textbf{F}(n)\}$ is convergent and $\lim\limits_{n\rightarrow +\infty}\textbf{F}(n)=\textbf{M}.$
	\end{proof}
	 %%%%%%%%%%%%%%%%%%%%%%%%%%%%%%%%%%%%%%%%%%%%%%%%%%%%%%%%%%%%%%%%%%%%%%%%%%%%%%%%%%%%%%%%%%%%%%%%%%%%%%%%%%%%%%%%%%%%%%%%%%%%%%%%%%%%%%%%%%%%%%%%%%%%%%%%%%%%%%%
	\begin{definition}(\emph{Limit inferior and limit superior of a sequence in ${I(\mathbb{R})}$}). Let $\{\textbf{F}(n)\}$ be a sequence. The limit inferior of $\{\textbf{F}(n)\}$, denoted $ \liminf\textbf{F}(n)$, is defined by
		\[\liminf\textbf{F}(n)=\lim\limits_{n\rightarrow +\infty}\inf\{\textbf{F}(n),\textbf{F}(n+1),\textbf{F}(n+2), \cdots\}.\]
		Similarly, limit superior of $\{\textbf{F}(n)\}$ is defined by
		\[\limsup\textbf{F}(n)=\lim\limits_{n\rightarrow +\infty}\sup\{\textbf{F}(n),\textbf{F}(n+1),\textbf{F}(n+2),\cdots\}.\]
	\end{definition}
	 %%%%%%%%%%%%%%%%%%%%%%%%%%%%%%%%%%%%%%%%%%%%%%%%%%%%%%%%%%%%%%%%%%%%%%%%%%%%%%%%%%%%%%%%%%%%%%%%%%%%%%%%%%%%%%%%%%%%%%%%%%%%%%%%%%%%%%%%%%%%%%%%%%%%%%%%%%%%%%%
	%\todo[inline, color=red!40]{Newly added example}
	\begin{example}
		Consider the following sequence in   ${I(\mathbb{R})}$: $$ \textbf{F}(n)=
		\begin{cases*}
		\left[\frac{1}{n^2},\frac{1}{n^2}+1\right] &\text{if}~ $n$~\text{is odd}\\
		[n, n^2+1] & \text{if}~ $n$~\text{is even}.
		\end{cases*}$$
		It is easy to see that $\inf\limits_{n\in\mathbb{N}} \left[\frac{1}{n^2},\frac{1}{n^2}+1\right]=[0,1]$ and $\inf\limits_{n\in\mathbb{N}} \left[n, n^2+1\right]=[1,2]$.  Therefore,
		%\begin{eqnarray*}
		%&&\lim\limits_{m\rightarrow +\infty}\inf\{\textbf{F}(m),\textbf{F}(m+1),\textbf{F}(m+2),\cdots\}=[0,1],\\&\text{and hence}&~\liminf\textbf{F}(m)=[0,1].
		%\end{eqnarray*}
		\[\lim\limits_{n\rightarrow +\infty}\inf\{\textbf{F}(n),\textbf{F}(n+1),\textbf{F}(n+2),\cdots\}=[0,1]~\text{and hence},~ \liminf\textbf{F}(n)=[0,1].\]
		Note that $\sup\limits_{n\in\mathbb{N}} \left[\frac{1}{n^2},\frac{1}{n^2}+1\right]=[1,2]$ and $\sup\limits_{n\in\mathbb{N}} \left[n, n^2+1\right]= +\infty$.  Thus,
		\[\lim\limits_{n\rightarrow +\infty}\sup\{\textbf{F}(n),\textbf{F}(n+1),\textbf{F}(n+2),\cdots\}=+\infty~\text{and hence},~\limsup\textbf{F}(n)=+\infty.\]
	\end{example}
	\section{$gH$-continuity and $gH$-semicontinuity of Interval-valued Functions}\label{sec3}
	In this section, we define $gH$-lower and $gH$-upper semicontinuity for extended IVFs and show that $gH$-continuity of an IVF implies $gH$-lower and upper semicontinuity and vice-versa. Further, we give a characterization of $gH$-lower semicontinuity in terms of the level sets of the IVF (Theorem \ref{n1}) and use this to prove that an extended $gH$-lower semicontinuous, level-bounded and proper IVF attains its minimum (Theorem \ref{g85}). We also give a characterization of the set argument minimum of an IVF (Theorem \ref{n2}).
	%\todo[inline, color=red!40]{In this section almost all the proofs are updated, so please check their mathematics as well}
%	\subsection{Sequences in ${I(\mathbb{R})}$}
	%	Before actually jumping into semicontinuity and the other analysis aspects for IVFs. We need to extend the idea of sequences for IVFs, which we do in this subsection.

	 %%%%%%%%%%%%%%%%%%%%%%%%%%%%%%%%%%%%%%%%%%%%%%%%%%%%%%%%%%%%%%%%%%%%%%%%%%%%%%%%%%%%%%%%%%%%%%%%%%%%%%%%%%%%%%%%%%%%%%%%%%%%%%%%%%%%%%%%%%%%%%%%%%%%%%%%%%%%%%%
	%\subsection{$gH$-lower and $gH$-upper Semicontinuity}
	Throughout this section, an extended IVF is an IVF wih domain $\mathcal{X}$ and codomain $\overline{I(\mathbb{R})}.$
	\begin{definition}(\emph{$gH$-limit of an IVF}).
		Let $\textbf{F}:\mathcal{S}\rightarrow I(\mathbb{R})$ be an IVF on a nonempty subset $\mathcal{S}$ of $\mathcal{X}$. The function $\textbf{F}$ is called tending to a limit $\textbf{L}\in I(\mathbb{R})$ as $x$ tends to $\bar{x}$, denoted by $\lim\limits_{x\rightarrow\bar{x}}\textbf{F}(x)$, if for each $\epsilon>0,~\text{there exists a}~\delta>0$ such that
		\begin{equation*}
		\lVert\textbf{F}(x)\ominus_{gH}\textbf{L}\rVert_{I(\mathbb{R})}<\epsilon~\text{whenever}~ 0<\lVert x-\bar{x}\rVert_{\mathcal{X}}<\delta.
		\end{equation*}
	\end{definition}
	 %%%%%%%%%%%%%%%%%%%%%%%%%%%%%%%%%%%%%%%%%%%%%%%%%%%%%%%%%%%%%%%%%%%%%%%%%%%%%%%%%%%%%%%%%%%%%%%%%%%%%%%%%%%%%%%%%%%%%%%%%%%%%%%%%%%%%%%%%%%%%%%%%%%%%%%%%%%%%%%
	\begin{definition}(\emph{$gH$-continuity}). Let $\textbf{F}:\mathcal{S}\rightarrow I(\mathbb{R})$ be an IVF on a nonempty subset $\mathcal{S}$ of $\mathcal{X}$. The function $\textbf{F}$ is said to be $gH$-continuous at $\bar{x}\in\mathcal{S}$ if for each $\epsilon>0,~\text{there exists a}~\delta>0$ such that
		\begin{equation*}
		\lVert\textbf{F}(x)\ominus_{gH}\textbf{F}(\bar{x})\rVert_{I(\mathbb{R})}<\epsilon~\text{whenever}~\lVert x-\bar{x}\rVert_{\mathcal{X}}<\delta.
		\end{equation*}
	\end{definition}
	 %%%%%%%%%%%%%%%%%%%%%%%%%%%%%%%%%%%%%%%%%%%%%%%%%%%%%%%%%%%%%%%%%%%%%%%%%%%%%%%%%%%%%%%%%%%%%%%%%%%%%%%%%%%%%%%%%%%%%%%%%%%%%%%%%%%%%%%%%%%%%%%%%%%%%%%%%%%%%%%
	 %%%%%%%%%%%%%%%%%%%%%%%%%%%%%%%%%%%%%%%%%%%%%%%%%%%%%%%%%%%%%%%%%%%%%%%%%%%%%%%%%%%%%%%%%%%%%%%%%%%%%%%%%%%%%%%%%%%%%%%%%%%%%%%%%%%%%%%%%%%%%%%%%%%%%%%%%%%%%%%
	\begin{definition}\label{g5}(\emph{Lower limit and $gH$-lower semicontinuity of an extended IVF}).
		The lower limit of an extended IVF $\textbf{F}$ at $\bar{x}\in\mathcal{X}$, denoted  $\liminf\limits_{x\rightarrow \bar{x}}\textbf{F}(x)$, is defined  by
		\begin{eqnarray*}\label{g6}
			\liminf_{x\rightarrow \bar{x}}\textbf{F}(x)&=&\lim\limits_{\delta\downarrow 0}\left(\inf\{\textbf{F}(x):x\in B_{\delta}(\bar{x})\}\right)\nonumber\\&=&\sup_{\delta>0}\left(\inf\{\textbf{F}(x):x\in B_{\delta}(\bar{x})\}\right).
		\end{eqnarray*}
		$\textbf{F}$ is called $gH$-lower semicontinuous ($gH$-lsc) at a point $\bar{x}$ if
		\begin{equation}\label{g7}
		\textbf{F}(\bar{x})\preceq\liminf_{x\rightarrow \bar{x}}\textbf{F}(x).
		\end{equation}
	Further, \textbf{F} is called $gH$-lsc on $\mathcal{X}$ if (\ref{g7}) holds for every $\bar{x}\in \mathcal{X}$.
	\end{definition}
	 %%%%%%%%%%%%%%%%%%%%%%%%%%%%%%%%%%%%%%%%%%%%%%%%%%%%%%%%%%%%%%%%%%%%%%%%%%%%%%%%%%%%%%%%%%%%%%%%%%%%%%%%%%%%%%%%%%%%%%%%%%%%%%%%%%%%%%%%%%%%%%%%%%%%%%%%%%%%%%%
	%\todo[inline, color=red!40]{Newly added example}
	\begin{example}
		Consider the following IVF $\textbf{F}: {\mathbb{R}}^2\rightarrow I(\mathbb{R})$:\\
		$$\textbf{F}(x_1,x_2)=
		\begin{cases*}
		[1,2]\odot\sin\left(\frac{1}{x_1}\right)\oplus\cos^2x_2 & \text{if}~ $x_1x_2\neq 0$\\
		[-2,-1] & \text{if}~ $x_1x_2= 0.$
		\end{cases*}$$
		The lower limit of $\textbf{F}$ at $(0,0)$ is given by
		\begin{equation}
		\liminf_{(x_1,x_2)\rightarrow (0,0)}\textbf{F}(x_1,x_2)=\lim\limits_{\delta\downarrow 0}\left(\inf\{\textbf{F}(x_1,x_2):(x_1,x_2)\in B_{\delta}(0,0)\}\right).\nonumber
		\end{equation}
		%\\&=&\lim\limits_{\delta\rightarrow 0}(\inf\{\textbf{C}_{1}\odot (x_1-3)^2\oplus \textbf{C}_{2}\odot (x_2-2)^2:(x_1,x_2)\in N_{\delta}(0,0)\})\\&=& \textbf{C}_{1}\odot 9\oplus \textbf{C}_{2}\odot 4.\\
		Note that as $x_1\rightarrow 0,~ \sin\left(\frac{1}{x_1}\right)$ oscillates between $-1$ and $1$. Therefore, for any $\delta>0$,  $$\inf_{(x_1,x_2)\in B_{\delta}(0,0)}\textbf{F}(x_1,x_2)=[1,2]\odot (-1)=[-2,-1].$$ Also, note that when $(x_1,x_2)=(0,0),~\textbf{F}(x_1,x_2)=[-2,-1].$ Thus,
		\begin{equation*}
		\liminf\limits_{(x_1,x_2)\rightarrow (0,0)}\textbf{F}(x_1,x_2)=[-2,-1].
		\end{equation*}
		Since $\textbf{F}(0,0)=[-2,-1]\preceq [-2,-1]=\liminf\limits_{(x_1,x_2)\rightarrow (0,0)}\textbf{F}(x_1,x_2)$, the function $\textbf{F}$ is $gH$-lsc at $(0,0)$.
	\end{example}
	 %%%%%%%%%%%%%%%%%%%%%%%%%%%%%%%%%%%%%%%%%%%%%%%%%%%%%%%%%%%%%%%%%%%%%%%%%%%%%%%%%%%%%%%%%%%%%%%%%%%%%%%%%%%%%%%%%%%%%%%%%%%%%%%%%%%%%%%%%%%%%%%%%%%%%%%%%%%%%%%
	\begin{note}\label{g75} Let $\textbf{F}$ be an extended IVF with $\textbf{F}(x)=\left[\underline{f}(x),\overline{f}(x)\right],$ where $\underline{f},~\overline{f}:\mathcal{X}\rightarrow\mathbb{R}\cup\{-\infty,+\infty\}$ be two extended real-valued functions. Then, $\textbf{F}$ is $gH$-lsc at $\bar{x}\in\mathcal{X}$ if and only if $\underline{f}$ and $\overline{f}$ both are lsc at $\bar{x}$. The reason is as follows.
		\begin{eqnarray*}
			\underline{f}~ \text{and}~ \overline{f}~ \text{are lsc at}~\bar{x}&\iff&\underline{f}(\bar{x})\leq\liminf_{x\rightarrow \bar{x}}\underline{f}(x)~\text{and}~\overline{f}(\bar{x})\leq\liminf_{x\rightarrow \bar{x}}\overline{f}(x)\\&\iff& \left[\underline{f}(\bar{x}),\overline{f}(\bar{x})\right]\preceq\left[\liminf_{x\rightarrow \bar{x}}\underline{f}(x),\liminf_{x\rightarrow \bar{x}}\overline{f}(x)\right]\\&\iff&\left[\underline{f}(\bar{x}),\overline{f}(\bar{x})\right]\preceq\liminf_{x\rightarrow \bar{x}}\left[\underline{f}(x),\overline{f}(x)\right],~\text{by Remark \ref{g54}}\\&\text{i.e.,}& \textbf{F}(\bar{x})\preceq\liminf_{x\rightarrow \bar{x}}\textbf{F}(x).
		\end{eqnarray*}	
	\end{note}
	 %%%%%%%%%%%%%%%%%%%%%%%%%%%%%%%%%%%%%%%%%%%%%%%%%%%%%%%%%%%%%%%%%%%%%%%%%%%%%%%%%%%%%%%%%%%%%%%%%%%%%%%%%%%%%%%%%%%%%%%%%%%%%%%%%%%%%%%%%%%%%%%%%%%%%%%%%%%%%%%
	Note \ref{g75} reduces our efforts to check $gH$-lower semicontinuity of extended IVFs that are given in the form  $\textbf{F}(x)=\left[\underline{f}(x),\overline{f}(x)\right]$. For example, consider $\textbf{F}:{\mathbb{R}}^2\rightarrow I(\mathbb{R})$ as
	$$\textbf{F}(x_1,x_2)=
		\begin{cases*}
		\left[\tfrac{\lvert  x_1x_2\rvert}{{2x_1}^2+{x_2}^2},\tfrac{e^{\lvert 6x_1x_2\rvert}}{{x_1}^2+{x_2}^2}\right]& \text{if}~ $x_1x_2\neq 0$\\
		\textbf{0}& \text{if}~ $x_1x_2= 0$
		\end{cases*}$$
	%\[\textbf{F}(x_1,x_2)=\left[{x_1}^2-2x_2+1,{x_1}^2+4{x_2}^2+e^{{x_1}^2+x_2}\right].\]
and take $\bar{x}=(0,0)$. It is easy to see that both
	$$\underline{f}(x_1,x_2)=	\begin{cases*}
		\tfrac{\lvert  x_1x_2\rvert}{2{x_1}^2+{x_2}^2}& \text{if}~ $x_1x_2\neq 0$\\
		0& \text{if}~ $x_1x_2= 0$
		\end{cases*}$$
		and	$$\overline{f}(x_1,x_2)=	\begin{cases*}
		\tfrac{e^{\lvert 6x_1x_2\rvert}}{{x_1}^2+{x_2}^2}& \text{if}~ $x_1x_2\neq 0$\\
		0& \text{if}~ $x_1x_2= 0$
		\end{cases*}$$
		are lsc at $\bar{x}$. Thus, by Note \ref{g75}, $\textbf{F}$ is $gH$-lsc at $\bar{x}$.
	%${x_1}^2-2x_2+1$ and ${x_1}^2+4{x_2}^2+e^{{x_1}^2+x_2}$ are lsc on ${\mathbb{R}}^2$. Thus, by Note \ref{g75}, $\textbf{F}$ is $gH$-lsc on ${\mathbb{R}}^2.$
	 %%%%%%%%%%%%%%%%%%%%%%%%%%%%%%%%%%%%%%%%%%%%%%%%%%%%%%%%%%%%%%%%%%%%%%%%%%%%%%%%%%%%%%%%%%%%%%%%%%%%%%%%%%%%%%%%%%%%%%%%%%%%%%%%%%%%%%%%%%%%%%%%%%%%%%%%%%%%%%%
	\begin{theorem}\label{g48}
		%Let $\mathcal{S}$ be a nonempty subset of $\mathcal{X}$
		Let $\textbf{F}$ be an extended IVF. Then, $\textbf{F}$ is $gH$-lsc at $\bar{x}\in\mathcal{X}$ if and only if for each $\epsilon>0$, there exists a $\delta>0$ such that $\textbf{F}(\bar{x})\ominus_{gH}[\epsilon,\epsilon]\prec\textbf{F}(x)~\text{for all}~ x\in B_{\delta}(\bar{x}).$
	\end{theorem}
	 %%%%%%%%%%%%%%%%%%%%%%%%%%%%%%%%%%%%%%%%%%%%%%%%%%%%%%%%%%%%%%%%%%%%%%%%%%%%%%%%%%%%%%%%%%%%%%%%%%%%%%%%%%%%%%%%%%%%%%%%%%%%%%%%%%%%%%%%%%%%%%%%%%%%%%%%%%%%%%%
	\begin{proof}
		Let $\textbf{F}$ be $gH$-lsc at $\bar{x}$.\\
		
		To the contrary, suppose there exists an $\epsilon_{0}>0$ such that for all $\delta>0$, $\textbf{F}(\bar{x})\ominus_{gH}[\epsilon_{0},\epsilon_{0}]\nprec\textbf{F}(x)$ for atleast one $x$ in  $B_{\delta}(\bar{x})$.\\
		
		Then,
		\begin{eqnarray*}
		&&\textbf{F}(\bar{x})\ominus_{gH}[\epsilon_{0},\epsilon_{0}]\nprec\inf\{\textbf{F}(x):x\in B_{\delta}(\bar{x})\}~\text{for all $\delta>0$}\\&\implies& \textbf{F}(\bar{x})\ominus_{gH}[\epsilon_{0},\epsilon_{0}]\nprec\lim\limits_{\delta\downarrow 0}(\inf\{\textbf{F}(x):x\in B_{\delta}(\bar{x})\})\\&\implies& \textbf{F}(\bar{x})\ominus_{gH}[\epsilon_{0},\epsilon_{0}]\nprec \liminf\limits_{x\rightarrow \bar{x}}\textbf{F}(x)\\&\implies& \textbf{F}(\bar{x})\npreceq \liminf\limits_{x\rightarrow \bar{x}}\textbf{F}(x),  ~\text{by (\ref{8}) of Lemma \ref{g28}},
		\end{eqnarray*}
	which contradicts that $\textbf{F}$ is $gH$-lsc at $\bar{x}$. Thus, for each $\epsilon>0$, there exists a $\delta>0$ such that $\textbf{F}(\bar{x})\ominus_{gH}[\epsilon,\epsilon]\prec\textbf{F}(x)~\text{for all}~ x\in B_{\delta}(\bar{x}).$ \\
		
		Conversely, suppose for a given $\epsilon>0$, there exists a $\delta>0$ such that $\textbf{F}(\bar{x})\ominus_{gH}[\epsilon,\epsilon]\prec\textbf{F}(x)~\text{for all}~ x\in B_{\delta}(\bar{x})$. Then, 		
		\begin{eqnarray*}
		&&\textbf{F}(\bar{x})\ominus_{gH}[\epsilon,\epsilon]\preceq\inf\{\textbf{F}(x):x\in B_{\delta}(\bar{x})\}\\
		&\implies&\textbf{F}(\bar{x})\ominus_{gH}[\epsilon,\epsilon]\preceq\lim\limits_{\delta\downarrow 0}(\inf\{\textbf{F}(x):x\in B_{\delta}(\bar{x})\})\\
		&\implies&\textbf{F}(\bar{x})\ominus_{gH}[\epsilon,\epsilon]\preceq \liminf\limits_{x\rightarrow \bar{x}}\textbf{F}(x).
		\end{eqnarray*}
		As, $\textbf{F}(\bar{x})\ominus_{gH}[\epsilon,\epsilon]\preceq \liminf\limits_{x\rightarrow \bar{x}}\textbf{F}(x)$ for every $\epsilon>0$, we have $\textbf{F}(\bar{x})\preceq \liminf\limits_{x\rightarrow \bar{x}}\textbf{F}(x)$. Thus, $\textbf{F}~\text{ is}~ gH\text{-lsc at}~ \bar{x}.$
	\end{proof}
	 %%%%%%%%%%%%%%%%%%%%%%%%%%%%%%%%%%%%%%%%%%%%%%%%%%%%%%%%%%%%%%%%%%%%%%%%%%%%%%%%%%%%%%%%%%%%%%%%%%%%%%%%%%%%%%%%%%%%%%%%%%%%%%%%%%%%%%%%%%%%%%%%%%%%%%%%%%%%%%%
	\begin{definition}\label{g8}(\emph{Upper limit and $gH$-upper semicontinuity of an extended IVF}). %Let $\mathcal{S}$ be a nonempty subset of $\mathcal{X}$.
		The upper limit of an extended IVF $\textbf{F}$ at $\bar{x}\in\mathcal{X}$, denoted  $\limsup\limits_{x\rightarrow \bar{x}}\textbf{F}(x)$, is defined as
		\begin{eqnarray*}\label{g9}
			\limsup_{x\rightarrow \bar{x}}\textbf{F}(x)&=&\lim\limits_{\delta\downarrow 0}\left(\sup\{\textbf{F}(x):x\in B_{\delta}(\bar{x})\}\right)\nonumber\\&=&\inf_{\delta>0}\left(\sup\{\textbf{F}(x):x\in B_{\delta}(\bar{x})\}\right).
		\end{eqnarray*}
		$\textbf{F}$ is called $gH$-upper semicontinuous ($gH$-usc) at $\bar{x}$ if
		\begin{equation}\label{g10}
		\limsup_{x\rightarrow \bar{x}}\textbf{F}(x)\preceq\textbf{F}(\bar{x}).
		\end{equation}
		Further, \textbf{F} is called $gH$-usc on $\mathcal{X}$ if (\ref{g10}) holds for every $\bar{x}\in \mathcal{X}$.
	\end{definition}
	 %%%%%%%%%%%%%%%%%%%%%%%%%%%%%%%%%%%%%%%%%%%%%%%%%%%%%%%%%%%%%%%%%%%%%%%%%%%%%%%%%%%%%%%%%%%%%%%%%%%%%%%%%%%%%%%%%%%%%%%%%%%%%%%%%%%%%%%%%%%%%%%%%%%%%%%%%%%%%%%

	 %%%%%%%%%%%%%%%%%%%%%%%%%%%%%%%%%%%%%%%%%%%%%%%%%%%%%%%%%%%%%%%%%%%%%%%%%%%%%%%%%%%%%%%%%%%%%%%%%%%%%%%%%%%%%%%%%%%%%%%%%%%%%%%%%%%%%%%%%%%%%%%%%%%%%%%%%%%%%%%
	\begin{note}
		Let $\textbf{F}$ be an extended IVF with $\textbf{F}(x)=\left[\underline{f}(x),\overline{f}(x)\right],$ where $\underline{f},~\overline{f}:\mathcal{X}\rightarrow\mathbb{R}\cup\{-\infty,+\infty\}$ be two extended real-valued functions. Then, because of a similar reason as in Note \ref{g75}, $\textbf{F}$ is $gH$-usc at $\bar{x}\in\mathcal{X}$ if and only if $\underline{f}$ and $\overline{f}$ are usc at $\bar{x}$.
	\end{note}
	 %%%%%%%%%%%%%%%%%%%%%%%%%%%%%%%%%%%%%%%%%%%%%%%%%%%%%%%%%%%%%%%%%%%%%%%%%%%%%%%%%%%%%%%%%%%%%%%%%%%%%%%%%%%%%%%%%%%%%%%%%%%%%%%%%%%%%%%%%%%%%%%%%%%%%%%%%%%%%%%
	\begin{theorem}\label{g49}%Let $\mathcal{S}$ be a nonempty subset of $\mathcal{X}$ and
		Let $\textbf{F}$ be an extended IVF. Then, $\textbf{F}$ is $gH$-usc at $\bar{x}\in\mathcal{X}$ if and only if for each $\epsilon>0$, there exists a $\delta>0$ such that $\textbf{F}(x)\prec\textbf{F}(\bar{x})\oplus[\epsilon,\epsilon]~\text{for all}~ x\in B_{\delta}(\bar{x}).$
	\end{theorem}
	 %%%%%%%%%%%%%%%%%%%%%%%%%%%%%%%%%%%%%%%%%%%%%%%%%%%%%%%%%%%%%%%%%%%%%%%%%%%%%%%%%%%%%%%%%%%%%%%%%%%%%%%%%%%%%%%%%%%%%%%%%%%%%%%%%%%%%%%%%%%%%%%%%%%%%%%%%%%%%%%
	\begin{proof}
	Similar to the proof of Theorem \ref{g48}.
	\end{proof}
	 %%%%%%%%%%%%%%%%%%%%%%%%%%%%%%%%%%%%%%%%%%%%%%%%%%%%%%%%%%%%%%%%%%%%%%%%%%%%%%%%%%%%%%%%%%%%%%%%%%%%%%%%%%%%%%%%%%%%%%%%%%%%%%%%%%%%%%%%%%%%%%%%%%%%%%%%%%%%%%%
	\begin{theorem}\label{g44} %Let $\mathcal{S}$ be a nonempty subset of $\mathcal{X}$.
		An extended IVF $\textbf{F}$ is $gH$-continuous if and only if $\textbf{F}$ is both $gH$-lower and upper semicontinuous.
	\end{theorem}
	 %%%%%%%%%%%%%%%%%%%%%%%%%%%%%%%%%%%%%%%%%%%%%%%%%%%%%%%%%%%%%%%%%%%%%%%%%%%%%%%%%%%%%%%%%%%%%%%%%%%%%%%%%%%%%%%%%%%%%%%%%%%%%%%%%%%%%%%%%%%%%%%%%%%%%%%%%%%%%%%
	\begin{proof}
		Let $\textbf{F}$ be $gH$-continuous at $\bar{x}\in\mathcal{X}.$ Then, for each $\epsilon>0$, there exists a $\delta>0$ such that
		\begin{eqnarray*}
		&&\lVert\textbf{F}(x)\ominus_{gH}\textbf{F}(\bar{x})\rVert_{I(\mathbb{R})}<\epsilon~\text{for all}~ x\in B_{\delta}(\bar{x})\\
		&\iff& \textbf{F}(\bar{x})\ominus_{gH}[\epsilon,\epsilon]\prec\textbf{F}(x)\prec\textbf{F}(\bar{x})\oplus[\epsilon,\epsilon]~\text{for all}~ x\in B_{\delta}(\bar{x}),~\text{by ({\ref{6}}) of Lemma \ref{g28}}\\
		&\iff& \textbf{F}(\bar{x})~ \text{is}~ gH\text{-lsc and}~ gH\text{-usc at}~ \bar{x},~\text{by Theorems \ref{g48} and \ref{g49}.}
		\end{eqnarray*}
	\end{proof}
	 %%%%%%%%%%%%%%%%%%%%%%%%%%%%%%%%%%%%%%%%%%%%%%%%%%%%%%%%%%%%%%%%%%%%%%%%%%%%%%%%%%%%%%%%%%%%%%%%%%%%%%%%%%%%%%%%%%%%%%%%%%%%%%%%%%%%%%%%%%%%%%%%%%%%%%%%%%%%%%%
	\begin{definition}(\emph{Proper IVF}). An extended IVF $\textbf{F}$ is called a proper function if there exists an $\bar x\in\mathcal{X}$ such that $\textbf{F}(\bar x)\prec [+\infty,+\infty]$ and $[-\infty,-\infty]\prec\textbf{F}(x)~\text{for all}~ x\in\mathcal{X}.$
	\end{definition}
	 %%%%%%%%%%%%%%%%%%%%%%%%%%%%%%%%%%%%%%%%%%%%%%%%%%%%%%%%%%%%%%%%%%%%%%%%%%%%%%%%%%%%%%%%%%%%%%%%%%%%%%%%%%%%%%%%%%%%%%%%%%%%%%%%%%%%%%%%%%%%%%%%%%%%%%%%%%%%%%%
	\begin{example}
		Consider the IVF $\textbf{F}:{\mathbb{R}}^2\rightarrow \overline{I(\mathbb{R})}$ that is given by  $\textbf{F}(x_1,x_2)=\left[x_1,{e}^{x_1}+{x_2}^2\right].$\\
		Note that $\textbf{F}(0,0)=[0,1]\prec [+\infty,+\infty]$. Also,   $[-\infty,-\infty]\prec\textbf{F}(x_1,x_2)~\text{for all}~ (x_1,x_2)\in {\mathbb{R}}^2.$
		Therefore, $\textbf{F}$ is a proper function.
	\end{example}
	 %%%%%%%%%%%%%%%%%%%%%%%%%%%%%%%%%%%%%%%%%%%%%%%%%%%%%%%%%%%%%%%%%%%%%%%%%%%%%%%%%%%%%%%%%%%%%%%%%%%%%%%%%%%%%%%%%%%%%%%%%%%%%%%%%%%%%%%%%%%%%%%%%%%%%%%%%%%%%%%
	\begin{lemma}\label{g79}
		Let $\mathbf{F}_1$ and $\mathbf{F}_2$ be two proper extended IVFs, and $\mathcal{S}$ be a nonempty subset of $\mathcal{X}$. Then,
		\begin{enumerate}[(i)]
			\item\label{g78} $\inf\limits_{x\in\mathcal{S}}\mathbf{F}_1(x)\oplus\inf\limits_{x\in\mathcal{S}}\mathbf{F}_2(x)\preceq\inf\limits_{x\in\mathcal{S}}\{\mathbf{F}_1(x)\oplus\mathbf{F}_2(x)\}$ and
			\item\label{g77} $\sup\limits_{x\in\mathcal{S}}\{\mathbf{F}_1(x)\oplus\mathbf{F}_2(x)\}\preceq\sup\limits_{x\in\mathcal{S}}\mathbf{F}_1(x)\oplus\sup\limits_{x\in\mathcal{S}}\mathbf{F}_2(x).$
		\end{enumerate}
	\end{lemma}
	 %%%%%%%%%%%%%%%%%%%%%%%%%%%%%%%%%%%%%%%%%%%%%%%%%%%%%%%%%%%%%%%%%%%%%%%%%%%%%%%%%%%%%%%%%%%%%%%%%%%%%%%%%%%%%%%%%%%%%%%%%%%%%%%%%%%%%%%%%%%%%%%%%%%%%%%%%%%%%%%
	\begin{proof}
		Let $\boldsymbol{\alpha}_1=\inf\limits_{x\in\mathcal{S}}\mathbf{F}_1(x)$ and $\boldsymbol{\alpha}_2=\inf\limits_{x\in\mathcal{S}}\mathbf{F}_2(x)$. Then,
		\begin{eqnarray*}
			&&\boldsymbol{\alpha}_1\preceq \mathbf{F}_1(x)~\text{for all}~ x\in\mathcal{S} ~\text{and}~ \boldsymbol{\alpha}_2\preceq\mathbf{F}_2(x)~\text{for all}~ x\in\mathcal{S}\\&\implies& \boldsymbol{\alpha}_1\oplus\boldsymbol{\alpha}_2\preceq\mathbf{F}_1(x)\oplus\mathbf{F}_2(x)~\text{for all}~ x\in\mathcal{S},~\text{by (\ref{2}) of Lemma \ref{g23}}\\&\implies& \boldsymbol{\alpha}_1\oplus\boldsymbol{\alpha}_2\preceq\inf\limits_{x\in\mathcal{S}}(\mathbf{F}_1(x)\oplus\mathbf{F}_2(x))\\&\text{i.e.,}&\inf\limits_{x\in\mathcal{S}}\mathbf{F}_1(x)\oplus\inf\limits_{x\in\mathcal{S}}\mathbf{F}_2(x)\preceq\inf\limits_{x\in\mathcal{S}}\{\mathbf{F}_1(x)\oplus\mathbf{F}_2(x)\}.
		\end{eqnarray*}
		Part (\ref{g77}) can be similarly proved.
	\end{proof}
	 %%%%%%%%%%%%%%%%%%%%%%%%%%%%%%%%%%%%%%%%%%%%%%%%%%%%%%%%%%%%%%%%%%%%%%%%%%%%%%%%%%%%%%%%%%%%%%%%%%%%%%%%%%%%%%%%%%%%%%%%%%%%%%%%%%%%%%%%%%%%%%%%%%%%%%%%%%%%%%%
	\begin{theorem}\label{g82}
		Let $\mathbf{F}_1$ and $\mathbf{F}_2$ be two proper extended IVFs, and $\mathcal{S}$ be a nonempty subset of $\mathcal{X}$. Then,
		\begin{enumerate}[(i)]
			\item\label{g80} $	\liminf\limits_{x\rightarrow\bar{x}}\mathbf{F}_1(x)\oplus \liminf\limits_{x\rightarrow\bar{x}}\mathbf{F}_2(x)\preceq	 \liminf\limits_{x\rightarrow \bar{x}}(\mathbf{F}_1\oplus \mathbf{F}_2)(x)$ and
			\item \label{g81}$ \limsup\limits_{x\rightarrow \bar{x}}(\mathbf{F}_1\oplus \mathbf{F}_2)(x)\preceq\limsup\limits_{x\rightarrow\bar{x}}\mathbf{F}_1(x)\oplus \limsup\limits_{x\rightarrow\bar{x}}\mathbf{F}_2(x)$.
		\end{enumerate}
	\end{theorem}
	 %%%%%%%%%%%%%%%%%%%%%%%%%%%%%%%%%%%%%%%%%%%%%%%%%%%%%%%%%%%%%%%%%%%%%%%%%%%%%%%%%%%%%%%%%%%%%%%%%%%%%%%%%%%%%%%%%%%%%%%%%%%%%%%%%%%%%%%%%%%%%%%%%%%%%%%%%%%%%%%
	\begin{proof}
		\begin{eqnarray*}
			\liminf\limits_{x\rightarrow\bar{x}}\mathbf{F}_1(x)\oplus \liminf\limits_{x\rightarrow\bar{x}}\mathbf{F}_2(x)&=&\lim\limits_{\delta\downarrow 0}\inf\limits_{x\in {B_{\delta}(\bar{x})}}\mathbf{F}_1(x)\oplus\lim\limits_{\delta\downarrow 0}\inf\limits_{x\in {B_{\delta}(\bar{x})}}\mathbf{F}_2(x),~\text{by Definition \ref{g5}}\\&\preceq&\lim\limits_{\delta\downarrow 0}\left(\inf\limits_{x\in {B_{\delta}(\bar{x})}}\mathbf{F}_1(x)\oplus \inf\limits_{x\in {B_{\delta}(\bar{x})}}\mathbf{F}_2(x)\right)\\&\preceq&\lim\limits_{\delta\downarrow 0}\inf\limits_{x\in {B_{\delta}(\bar{x})}}(\mathbf{F}_1\oplus\mathbf{F}_2)(x),~\text{by (\ref{g78}) of Lemma \ref{g79}}\\&=& \liminf\limits_{x\rightarrow \bar{x}}(\mathbf{F}_1\oplus \mathbf{F}_2)(x).
		\end{eqnarray*}
		This completes the proof of (\ref{g80}).\\
		Part (\ref{g81}) can be similarly proved.
	\end{proof}
	 %%%%%%%%%%%%%%%%%%%%%%%%%%%%%%%%%%%%%%%%%%%%%%%%%%%%%%%%%%%%%%%%%%%%%%%%%%%%%%%%%%%%%%%%%%%%%%%%%%%%%%%%%%%%%%%%%%%%%%%%%%%%%%%%%%%%%%%%%%%%%%%%%%%%%%%%%%%%%%%
	\begin{theorem}\label{g83}
		Let $\mathbf{F}_1$ and $\mathbf{F}_2$ be two proper and $gH$-lsc extended IVFs. Then, $\mathbf{F}_1\oplus\mathbf{F}_2$ is $gH$-lsc.
	\end{theorem}
	 %%%%%%%%%%%%%%%%%%%%%%%%%%%%%%%%%%%%%%%%%%%%%%%%%%%%%%%%%%%%%%%%%%%%%%%%%%%%%%%%%%%%%%%%%%%%%%%%%%%%%%%%%%%%%%%%%%%%%%%%%%%%%%%%%%%%%%%%%%%%%%%%%%%%%%%%%%%%%%%
	\begin{proof}
		Take $\bar{x}\in\mathcal{X}$. Since $\mathbf{F}_1$ and $\mathbf{F}_2$ are $gH$-lsc at $\bar{x}$, we have
		\begin{eqnarray*}
			 &&\mathbf{F}_1(\bar{x})\preceq\liminf\limits_{x\rightarrow\bar{x}}\mathbf{F}_1(x)~\text{and}~\mathbf{F}_2(\bar{x})\preceq\liminf\limits_{x\rightarrow\bar{x}}\mathbf{F}_2(x)\\&\implies& \mathbf{F}_1(\bar{x})\oplus\mathbf{F}_2(\bar{x})\preceq\liminf\limits_{x\rightarrow\bar{x}}\mathbf{F}_1(x)\oplus\liminf\limits_{x\rightarrow\bar{x}}\mathbf{F}_2(x),~\text{by (\ref{2}) of Lemma \ref{g23}}\\&\implies&(\mathbf{F}_1\oplus\mathbf{F}_2)(\bar{x})\preceq\liminf\limits_{x\rightarrow\bar{x}}(\mathbf{F}_1\oplus\mathbf{F}_2)(x),~\text{by (\ref{g80}) of Theorem \ref{g82}}\\&\implies& \mathbf{F}_1\oplus\mathbf{F}_2~\text{is} ~gH\text{-lsc at}~\bar{x}.
		\end{eqnarray*}
		Since $\bar{x}$ is arbitrarily chosen, so $\mathbf{F}_1\oplus\mathbf{F}_2$ is $gH$-lsc on $\mathcal{X}$.
	\end{proof}
	 %%%%%%%%%%%%%%%%%%%%%%%%%%%%%%%%%%%%%%%%%%%%%%%%%%%%%%%%%%%%%%%%%%%%%%%%%%%%%%%%%%%%%%%%%%%%%%%%%%%%%%%%%%%%%%%%%%%%%%%%%%%%%%%%%%%%%%%%%%%%%%%%%%%%%%%%%%%%%%%
	\begin{lemma}\label{g20}(\emph{Characterization of lower limits of IVFs}). Let $\textbf{F}$ be an extended IVF. Then,
		\begin{eqnarray*}
			\liminf_{x\rightarrow \bar{x}}\textbf{F}(x)=\inf \left\{\boldsymbol{\alpha}\in\overline{I(\mathbb{R})}:\text{there exists a sequence}~ x_k\rightarrow \bar{x}~\text{with}~\textbf{F}(x_k)\rightarrow \boldsymbol{\alpha}\right\}.
		\end{eqnarray*}
	\end{lemma}
	 %%%%%%%%%%%%%%%%%%%%%%%%%%%%%%%%%%%%%%%%%%%%%%%%%%%%%%%%%%%%%%%%%%%%%%%%%%%%%%%%%%%%%%%%%%%%%%%%%%%%%%%%%%%%%%%%%%%%%%%%%%%%%%%%%%%%%%%%%%%%%%%%%%%%%%%%%%%%%%%
	
\begin{proof}
		Let $\bar{\boldsymbol{\alpha}}=\liminf\limits_{x\rightarrow \bar{x}}\textbf{F}(x).$ Assume that sequence $x_k\rightarrow \bar{x}$ with $\textbf{F}(x_k)\rightarrow\boldsymbol{\alpha}$. In the below, we show that $\bar{\boldsymbol{\alpha}}\preceq \boldsymbol{\alpha}.$\\
		
		Since $x_k\rightarrow \bar{x}$, for any $\delta>0$, there exists $k_{\delta}\in\mathbb{N}$ such that $x_k\in B_{\delta}(\bar{x})$ for every $k\geq k_{\delta}$.\\
		
		Therefore,
		\begin{eqnarray*}
		&&\inf\{\textbf{F}(x):x\in B_{\delta}(\bar{x})\}\preceq\textbf{F}(x_k)~\text{for any}~\delta>0\\&\implies&\inf\{\textbf{F}(x):x\in B_{\delta}(\bar{x})\}\preceq\lim\limits_{k\rightarrow +\infty}\textbf{F}(x_k)~\text{for any}~\delta>0\\&\implies&\inf\{\textbf{F}(x):x\in B_{\delta}(\bar{x})\}\preceq\boldsymbol{\alpha}~\text{for any}~\delta>0\\&\implies& \lim\limits_{\delta\downarrow 0}\inf\{\textbf{F}(x):x\in B_{\delta}(\bar{x})\}\preceq\boldsymbol{\alpha}\\&\implies&  \liminf\limits_{x\rightarrow\bar{x}}\textbf{F}(x)=\bar{\boldsymbol{\alpha}}\preceq \boldsymbol{\alpha}.
		\end{eqnarray*}
		
		Next, we show that there exists a sequence  $x_k\rightarrow\bar{x}$ with $\textbf{F}(x_k)\rightarrow\bar{\boldsymbol{\alpha}}.$ \\
		
		Consider a nonnegative sequence $\{\delta_{k}\}$ with $\delta_{k}\downarrow 0$, and construct a sequence $\bar{\boldsymbol{\alpha}}_{k}=\inf\{\textbf{F}(x):x\in B_{\delta_{k}}(\bar{x})\}$.\\
		
		As $\delta_{k}\downarrow 0$, by Definition \ref{g5} of lower limit, $\bar{\boldsymbol{\alpha}}_{k}\rightarrow\bar{\boldsymbol{\alpha}}$. Also, by definition of infimum, for a given $\epsilon>0$ and $k\in\mathbb{N}$, there exists $x_k\in B_{\delta_{k}}(\bar{x})$ such that $\textbf{F}(x_k)\preceq \bar{\boldsymbol{\alpha}}_{k}.$ That is, $\bar{\boldsymbol{\alpha}}_{k}\preceq\textbf{F}(x_k)\preceq\boldsymbol{\alpha}_{k}$, where $\boldsymbol{\alpha}_{k}\rightarrow\bar{\boldsymbol{\alpha}}$.\\
		
		Note that $x_k\in B_{\delta_{k}}(\bar{x})$ and $\delta_{k}\downarrow 0.$ Therefore, as $k\rightarrow +\infty$, $x_k\rightarrow\bar{x}$. Also, note that $\textbf{F}(x_k)$ is a monotonic increasing bounded sequence and therefore, by Lemma \ref{p}, $\textbf{F}(x_k)$ converges to $\bar{\boldsymbol{\alpha}}$, and the proof is complete.
		
		%Since $\bar{\boldsymbol{\alpha}}_{k}\rightarrow\bar{\boldsymbol{\alpha}}$, so for each $\epsilon>0$, there exists an $m\in\mathbb{N}$ such that
		%\begin{equation}\label{g26}
		%\lVert \bar{\boldsymbol{\alpha}}_{k}\ominus_{gH}\bar{\boldsymbol{\alpha}}\rVert_{I(\mathbb{R})}<\frac{\epsilon}{2}~\text{for all}~k\geq m.
		%\end{equation}
		%Then,
		%\begin{eqnarray*}
		%	\lVert\textbf{F}(x_k)\ominus_{gH}\bar{\boldsymbol{\alpha}}\rVert_{I(\mathbb{R})}&\leq&\lVert (\textbf{F}(x_k)\ominus_{gH}\bar{\boldsymbol{\alpha}}_{k})\oplus(\bar{\boldsymbol{\alpha}}_{k}\ominus_{gH}\bar{\boldsymbol{\alpha}})\rVert_{I(\mathbb{R})}, \text{by Lemma \ref{g27}}\\&\leq& \lVert\textbf{F}(x_k)\ominus_{gH}\bar{\boldsymbol{\alpha}}_{k}\rVert_{I(\mathbb{R})}+\lVert \bar{\boldsymbol{\alpha}}_{k}\ominus_{gH}\bar{\boldsymbol{\alpha}}\rVert_{I(\mathbb{R})}, \text{by (\ref{1}) of Lemma \ref{g23}}\\&<& \frac{\epsilon}{2}+\frac{\epsilon}{2}~\text{for all}~k\geq m,~ \text{by equations (\ref{g25}) and (\ref{g26})}\\&=& \epsilon.
		%\end{eqnarray*}
	%	Thus, $\lVert\textbf{F}(x_k)\ominus_{gH}\bar{\boldsymbol{\alpha}}\rVert_{I(\mathbb{R})}<\epsilon~\text{for all}~k\geq m$. This implies $ \textbf{F}(x_k)\rightarrow \bar{\boldsymbol{\alpha}}$, and the proof is complete.
	\end{proof}
	 %%%%%%%%%%%%%%%%%%%%%%%%%%%%%%%%%%%%%%%%%%%%%%%%%%%%%%%%%%%%%%%%%%%%%%%%%%%%%%%%%%%%%%%%%%%%%%%%%%%%%%%%%%%%%%%%%%%%%%%%%%%%%%%%%%%%%%%%%%%%%%%%%%%%%%%%%%%%%%%
	\begin{lemma}\label{g50}(\emph{Characterization of upper limits of IVFs}). Let $\textbf{F}$ be an extended IVF. Then,
		\begin{eqnarray*}
			\limsup_{x\rightarrow \bar{x}}\textbf{F}(x)=\sup \left\{\boldsymbol{\alpha}\in \overline{I(\mathbb{R})}:\text{there exists a sequence}~ x_k\rightarrow \bar{x}~\text{with}~\textbf{F}(x_k)\rightarrow \boldsymbol{\alpha}\right\}.
		\end{eqnarray*}
	\end{lemma}
	 %%%%%%%%%%%%%%%%%%%%%%%%%%%%%%%%%%%%%%%%%%%%%%%%%%%%%%%%%%%%%%%%%%%%%%%%%%%%%%%%%%%%%%%%%%%%%%%%%%%%%%%%%%%%%%%%%%%%%%%%%%%%%%%%%%%%%%%%%%%%%%%%%%%%%%%%%%%%%%%
	\begin{proof}
		Similar to the proof of Lemma \ref{g20}.
	\end{proof}
	 %%%%%%%%%%%%%%%%%%%%%%%%%%%%%%%%%%%%%%%%%%%%%%%%%%%%%%%%%%%%%%%%%%%%%%%%%%%%%%%%%%%%%%%%%%%%%%%%%%%%%%%%%%%%%%%%%%%%%%%%%%%%%%%%%%%%%%%%%%%%%%%%%%%%%%%%%%%%%%%
	\begin{definition}\label{g16}(\emph{Level set of an IVF}). Let $\textbf{F}$ be an extended IVF. For an $\boldsymbol{\alpha}\in \overline{I(\mathbb{R})}$, the level set of $\textbf{F}$, denoted as lev$_{\boldsymbol{\alpha}\nprec}\textbf{F}$, is defined by \[
		\text{lev}_{\boldsymbol{\alpha}\nprec}\textbf{F}=\{x\in\mathcal{X}:\boldsymbol{\alpha}\nprec \textbf{F}(x)\}.\]
	\end{definition}
	 %%%%%%%%%%%%%%%%%%%%%%%%%%%%%%%%%%%%%%%%%%%%%%%%%%%%%%%%%%%%%%%%%%%%%%%%%%%%%%%%%%%%%%%%%%%%%%%%%%%%%%%%%%%%%%%%%%%%%%%%%%%%%%%%%%%%%%%%%%%%%%%%%%%%%%%%%%%%%%%
	\begin{example}
		Consider $\textbf{F}:{\mathbb{R}}^{2}\rightarrow \overline{I(\mathbb{R})}$ as $\textbf{F}(x)=[1,2]\odot{x_1}^2\oplus[3,4]\odot e^{{x_2}^2}$ and $\boldsymbol{\alpha}=[-1,10]$. Then,
		\begin{eqnarray*}
			\text{lev}_{\boldsymbol{\alpha}\nprec}\textbf{F}&=&\left\{(x_1,x_2)\in {\mathbb{R}}^2:[-1,10]\nprec [1,2]\odot{x_1}^2\oplus[3,4]\odot e^{{x_2}^2}\right\}\\&=&\left\{(x_1,x_2)\in {\mathbb{R}}^2:[-1,10]\nprec\left[{x_1}^2+3e^{{x_2}^2},2{x_1}^2+4e^{{x_2}^2}\right]\right\}\\&=&\Big\{(x_1,x_2)\in {\mathbb{R}}^2:\left[{x_1}^2+3e^{{x_2}^2},2{x_1}^2+4e^{{x_2}^2}\right]\preceq [-1,10] ~\text{or}\\
			&&~~~~~~~~~~~~~~~~~~~~~~~~~~~~[-1,10] ~\text{and}~ \Big[{x_1}^2+3e^{{x_2}^2},2{x_1}^2+4e^{{x_2}^2}\Big]~\text{are not comparable}\Big\}\\&=&\Big\{(x_1,x_2)\in {\mathbb{R}}^2:[-1,10] ~\text{and}~ \left[{x_1}^2+3e^{{x_2}^2},2{x_1}^2+4e^{{x_2}^2}\right]~\text{are not comparable}\Big\}\\&=&\Big\{(x_1,x_2)\in {\mathbb{R}}^2:`{x_1}^2+3e^{{x_2}^2}< -1~\text{and}~2{x_1}^2+4e^{{x_2}^2}>10\text{'}~\text{or}\\&&~~~~~~~~~~~~~~~~~~~~~~~~~~~~`{x_1}^2+3e^{{x_2}^2}>-1~\text{and}~2{x_1}^2+4e^{{x_2}^2}< 10\text{'}\Big\}\\&=&\left\{(x_1,x_2)\in {\mathbb{R}}^2:{x_1}^2+3e^{{x_2}^2}>-1~\text{and}~2{x_1}^2+4e^{{x_2}^2}< 10\right\}\\&=&\left\{(x_1,x_2)\in {\mathbb{R}}^2:2{x_1}^2+4e^{{x_2}^2}< 10\right\}.
		\end{eqnarray*}
		Hence,\[\text{lev}_{\boldsymbol{\alpha}\nprec}\textbf{F}=\left\{(x_1,x_2)\in {\mathbb{R}}^2:{x_1}^2+2e^{{x_2}^2}< 5\right\}.\]
	\end{example}
	 %%%%%%%%%%%%%%%%%%%%%%%%%%%%%%%%%%%%%%%%%%%%%%%%%%%%%%%%%%%%%%%%%%%%%%%%%%%%%%%%%%%%%%%%%%%%%%%%%%%%%%%%%%%%%%%%%%%%%%%%%%%%%%%%%%%%%%%%%%%%%%%%%%%%%%%%%%%%%%%
	\begin{definition}(\emph{Level-bounded IVF}). An extended IVF $\textbf{F}$ is said to be level-bounded if for any $\boldsymbol{\alpha}\in I(\mathbb{R})$, lev$_{\boldsymbol{\alpha}\nprec}\textbf{F}$ is bounded.
	\end{definition}
	 %%%%%%%%%%%%%%%%%%%%%%%%%%%%%%%%%%%%%%%%%%%%%%%%%%%%%%%%%%%%%%%%%%%%%%%%%%%%%%%%%%%%%%%%%%%%%%%%%%%%%%%%%%%%%%%%%%%%%%%%%%%%%%%%%%%%%%%%%%%%%%%%%%%%%%%%%%%%%%%
	\begin{lemma}\label{g68}
		Let $\textbf{F}$ be an extended IVF and $\bar{x}\in\mathcal{X}$. Then,
		\begin{equation}\label{g72}
		\inf\limits_{\{x_k\}}\left(\liminf\textbf{F}(x_k)\right)\nprec\liminf_{x\rightarrow \bar{x}}\textbf{F}(x),
		\end{equation}
		where the infimum on the left-hand side is taken over all sequences $x_k\rightarrow\bar{x}.$
	\end{lemma}
	 %%%%%%%%%%%%%%%%%%%%%%%%%%%%%%%%%%%%%%%%%%%%%%%%%%%%%%%%%%%%%%%%%%%%%%%%%%%%%%%%%%%%%%%%%%%%%%%%%%%%%%%%%%%%%%%%%%%%%%%%%%%%%%%%%%%%%%%%%%%%%%%%%%%%%%%%%%%%%%%
	\begin{proof}
		Let $\textbf{M}=\liminf\limits_{x\rightarrow \bar{x}}\textbf{F}(x)$ and $\textbf{L}=\inf\limits_{\{x_k\}} \liminf\textbf{F}(x_k)$.\\
		
		If $\textbf{M}=-\infty$, there is nothing to prove.\\
		
		Next, let $\textbf{M}=+\infty$. Let $\{x_k\}$ be an arbitrary sequence converging to $\bar{x}$. We show that $\textbf{F}(x_k)\rightarrow +\infty$. Since $\textbf{M}=+\infty$, for any given $\alpha>0$, there exists a $\delta>0$ such that $[\alpha,\alpha]\prec\inf_{x\in B_\delta(\bar{x})}\textbf{F}(x).$ Since $x_k\rightarrow\bar{x}$, there exists an integer $m>0$ such that $x_k\in B_\delta(\bar{x})~\text{for all}~ n\geq m$.
		Thus, $[\alpha,\alpha]\prec\textbf{F}(x_k)~\text{for all}~ n\geq m$, and hence  $\textbf{F}(x_k)\rightarrow +\infty$.\\
		
		Finally, let  $[-\infty,-\infty]\prec\textbf{M}\prec [+\infty,+\infty]$, i.e., $\textbf{M}\in I(\mathbb{R}).$ Let, if possible, there exists an $\epsilon_0>0$ such that for all $\delta>0$, $\inf_{x\in B_\delta(\bar{x})}\textbf{F}(x)\preceq \textbf{M}\ominus_{gH}[\epsilon_0,\epsilon_0]$. Then,
		\begin{eqnarray*}
		&&\lim\limits_{\delta\downarrow 0}\inf_{x\in B_\delta(\bar{x})}\textbf{F}(x)\preceq \textbf{M}\ominus_{gH}[\epsilon_0,\epsilon_0]\\&\implies&\liminf\limits_{x\rightarrow \bar{x}}\textbf{F}(x)\preceq \textbf{M}\ominus_{gH}[\epsilon_0,\epsilon_0]\\&\text{i.e.,}&\textbf{M}\preceq \textbf{M}\ominus_{gH}[\epsilon_0,\epsilon_0],
		\end{eqnarray*}
		which is not true. Thus, for a given $\epsilon>0,$ there exists a $\delta>0$ such that $\inf_{x\in B_\delta(\bar{x})}\textbf{F}(x)\npreceq \textbf{M}\ominus_{gH}[\epsilon,\epsilon]$. This implies $\textbf{F}(x)\npreceq\textbf{M}\ominus_{gH}[\epsilon,\epsilon]~\text{for all}~ x\in B_\delta(\bar{x}).$ \\

		Let $\{x_k\}$ be a sequence converging to $\bar{x}$. Since $x_k\in B_\delta(\bar{x})$ for large enough $k$, we have
		$ \liminf\textbf{F}(x_k)\npreceq\textbf{M}\ominus_{gH}[\epsilon,\epsilon]~\text{for any}~\epsilon>0$. Thus, $\liminf\textbf{F}(x_k)\nprec\textbf{M}$ for any sequence converging to $\bar{x}$, and hence $\textbf{L}\nprec\textbf{M}$. Therefore, (\ref{g72}) holds.
	\end{proof}
	 %%%%%%%%%%%%%%%%%%%%%%%%%%%%%%%%%%%%%%%%%%%%%%%%%%%%%%%%%%%%%%%%%%%%%%%%%%%%%%%%%%%%%%%%%%%%%%%%%%%%%%%%%%%%%%%%%%%%%%%%%%%%%%%%%%%%%%%%%%%%%%%%%%%%%%%%%%%%%%%
	%\begin{definition}\label{g17}(\emph{Epigraph of an IVF}). Epigraph of an extended IVF $\textbf{F}$ is defined by
		%\begin{center}
	%	epi $\textbf{F}=\{(x,\boldsymbol{\alpha})\in\mathcal{X}\times I(\mathbb{R}):\boldsymbol{\alpha}\nprec\textbf{F}(x)\}.$
		%\end{center}
%	\end{definition}
	 %%%%%%%%%%%%%%%%%%%%%%%%%%%%%%%%%%%%%%%%%%%%%%%%%%%%%%%%%%%%%%%%%%%%%%%%%%%%%%%%%%%%%%%%%%%%%%%%%%%%%%%%%%%%%%%%%%%%%%%%%%%%%%%%%%%%%%%%%%%%%%%%%%%%%%%%%%%%%%%
	\begin{theorem}\label{n1}
	Let $\textbf{F}$ be an extended IVF. Then,  \textbf{F} is $gH$-lsc on $\mathcal{X}$ if and only if the level set  lev$_{\boldsymbol{\alpha}\nprec}\textbf{F}$  is closed for every $\boldsymbol{\alpha}\in I(\mathbb{R})$.
	\end{theorem}
	\begin{proof}
	 Let \textbf{F} be $gH$-lsc on $\mathcal{X}$.  For a fixed $\boldsymbol{\alpha}\in I(\mathbb{R})$, suppose that $\{x_k\}\subseteq$ lev$_{\boldsymbol{\alpha}\nprec}\textbf{F}$ such that $x_k\rightarrow\bar{x}.$ Then,
	 \begin{eqnarray*}
	 &&\boldsymbol{\alpha}\nprec\textbf{F}(x_k)\\&\implies& \boldsymbol{\alpha}\nprec\liminf\textbf{F}(x_k)\\&\implies&\boldsymbol{\alpha}\nprec\liminf\limits_{x\rightarrow\bar{x}}\textbf{F}(x),~\text{by Lemma \ref{g68}}\\&\implies&\boldsymbol{\alpha}\nprec\textbf{F}(\bar{x})~\text{since}~\textbf{F}~ \text{is $gH$-lsc at}~ \bar{x}.
	 \end{eqnarray*}
	 Thus, $\bar{x}\in$ lev$_{\boldsymbol{\alpha}\nprec}\textbf{F}$, and hence lev$_{\boldsymbol{\alpha}\nprec}\textbf{F}$ is closed.\\
	
	 Since $\boldsymbol{\alpha}\in I(\mathbb{R})$ is arbitrarly chosen, lev$_{\boldsymbol{\alpha}\nprec}\textbf{F}$ is closed for every $\boldsymbol{\alpha}\in I(\mathbb{R})$.\\
	
	 Conversely, suppose the level set  lev$_{\boldsymbol{\alpha}\nprec}\textbf{F}$ is closed for every $\boldsymbol{\alpha}\in I(\mathbb{R})$. Fix an $\bar{x}\in\mathcal{X}$. To prove that $\textbf{F}$ is $gH$-lsc at $\bar{x}$, we need to show that  \[\textbf{F}(\bar{x})\preceq\liminf\limits_{x\rightarrow \bar{x}}\textbf{F}(x).\] \\
	
	 Let $\bar{\boldsymbol{\alpha}}= \liminf\limits_{x\rightarrow \bar{x}}\textbf{F}(x)$.
		The case of $\bar{\boldsymbol{\alpha}}=+\infty$ is trivial; so assume $\bar{\boldsymbol{\alpha}}\prec [+\infty,+\infty].$\\
		
		By Lemma \ref{g20}, there exists a sequence $x_k\rightarrow\bar{x}$ with $\textbf{F}(x_k)\rightarrow\bar{\boldsymbol{\alpha}}$. For any $\boldsymbol{\alpha}$ such that $\bar{\boldsymbol{\alpha}}\prec\boldsymbol{\alpha}$, it will eventually be true that $\boldsymbol{\alpha}\nprec\textbf{F}(x_k)$, or in other words, that $x_k\in$ lev$_{\boldsymbol{\alpha}\nprec}\textbf{F}$. Since lev$_{\boldsymbol{\alpha}\nprec}\textbf{F}$ is closed, $\bar{x}\in$ lev$_{\boldsymbol{\alpha}\nprec}\textbf{F}$. \\
		
		Thus, $\boldsymbol{\alpha}\nprec\textbf{F}(\bar{x})$ for every $\boldsymbol{\alpha}$ such that $\bar{\boldsymbol{\alpha}}\prec\boldsymbol{\alpha}$, then $\bar{\boldsymbol{\alpha}}\nprec\textbf{F}(\bar{x})$. Therefore, either $\textbf{F}(\bar{x})\preceq\bar{\boldsymbol{\alpha}}$ or $\bar{\boldsymbol{\alpha}}$ and $\textbf{F}(\bar{x})$ are not comparable. But since $\bar{\boldsymbol{\alpha}}=\liminf\limits_{x\rightarrow \bar{x}}\textbf{F}(x)$, so $\bar{\boldsymbol{\alpha}}$ is comparable with $\textbf{F}(\bar{x})$, and hence  $\textbf{F}(\bar{x})\preceq\bar{\boldsymbol{\alpha}}$. \\
		
		Since $\bar{x}\in\mathcal{X}$ is arbitrarily chosen,  $\textbf{F}$ is $gH$-lsc on $\mathcal{X}.$ This completes the proof.
	\end{proof}
	\begin{definition}(\emph{Indicator function}). Consider a subset $\mathcal{S}$ of $\mathcal{X}.$ The indicator function of $\mathcal{S}$ is defined by\[ \delta_{\mathcal{S}}(s)=
		\begin{cases}
		\textbf{0}~~~~~~\text{if}~ s\in\mathcal{S}\\
		+\infty~\text{if}~s\notin \mathcal{S}.
		\end{cases}\]
	\end{definition}
	 %%%%%%%%%%%%%%%%%%%%%%%%%%%%%%%%%%%%%%%%%%%%%%%%%%%%%%%%%%%%%%%%%%%%%%%%%%%%%%%%%%%%%%%%%%%%%%%%%%%%%%%%%%%%%%%%%%%%%%%%%%%%%%%%%%%%%%%%%%%%%%%%%%%%%%%%%%%%%%%
	\begin{remark}\label{g86}
		\begin{enumerate}[(i)]
			\item It is easy to see that $\delta_{\mathcal{S}}$ is proper if and only if $\mathcal{S}$ is nonempty.
			\item By Theorem \ref{n1}, $\delta_{\mathcal{S}}$ is $gH$-lsc if and only if $\mathcal{S}$ is closed.
		\end{enumerate}
	\end{remark}
	 %%%%%%%%%%%%%%%%%%%%%%%%%%%%%%%%%%%%%%%%%%%%%%%%%%%%%%%%%%%%%%%%%%%%%%%%%%%%%%%%%%%%%%%%%%%%%%%%%%%%%%%%%%%%%%%%%%%%%%%%%%%%%%%%%%%%%%%%%%%%%%%%%%%%%%%%%%%%%%%
	\begin{definition}(\emph{Argument minimum of an IVF}). Let $\textbf{F}$ be an extended IVF. Then, the argument minimum of $\textbf{F}$, denoted as $\argmin\limits_{x\in\mathcal{X}}{\mathbf{F}(x)}$, is defined by $$\argmin\limits_{x\in\mathcal{X}}{\mathbf{F}(x)}=
		\begin{cases}
		\Big\{x\in \mathcal{X}:\mathbf{F}(x)=\inf\limits_{y\in\mathcal{X}}\mathbf{F}(y)\Big\} ~~\text{if}~ \inf\limits_{y\in\mathcal{X}}\mathbf{F}(y)\neq +\infty\\
		\emptyset~~~~~~~~~~~~~~~~~~~~~~~~~~~~~~~~~~~~~~~~~~~~~~~\text{if}~\inf\limits_{y\in\mathcal{X}}\mathbf{F}(y)=+\infty.
		\end{cases}$$
	\end{definition}
	 %%%%%%%%%%%%%%%%%%%%%%%%%%%%%%%%%%%%%%%%%%%%%%%%%%%%%%%%%%%%%%%%%%%%%%%%%%%%%%%%%%%%%%%%%%%%%%%%%%%%%%%%%%%%%%%%%%%%%%%%%%%%%%%%%%%%%%%%%%%%%%%%%%%%%%%%%%%%%%%
	\begin{example}
		Consider $\textbf{F}:{\mathbb{R}}^2\rightarrow \overline{I(\mathbb{R})}$ as $\textbf{F}(x_1,x_2)=
		\begin{cases*}
	\left[-\frac{1}{\lvert x_1\rvert},e^{-\frac{1}{\lvert x_1\rvert}+{x_2}^2}\right] & \text{if}~ $x_1\neq 0$\\
	[-\infty,0] & \text{if}~ $x_1= 0.$
		\end{cases*}$
		Then, $\inf\limits_{(x_1,x_2)\in{\mathbb{R}}^2}\mathbf{F}(x_1,x_2)=[-\infty,0].$\\
		\begin{eqnarray*}
			 \argmin\limits_{x\in{\mathbb{R}}^2}\mathbf{F}(x)&=&\left\{(x_1,x_2)\in{\mathbb{R}^2}:\textbf{F}(x_1,x_2)=\inf\limits_{x\in{\mathbb{R}}^2}\mathbf{F}(x_1,x_2)=[-\infty,0]\right\}\\&=&\{(0,x_2):x_2\in\mathbb{R}\}.
		\end{eqnarray*}
		Therefore, $\argmin\limits_{x\in{\mathbb{R}}^2}\mathbf{F}(x)=\{(0,x_2):x_2\in\mathbb{R}\}.$
	\end{example}
	 %%%%%%%%%%%%%%%%%%%%%%%%%%%%%%%%%%%%%%%%%%%%%%%%%%%%%%%%%%%%%%%%%%%%%%%%%%%%%%%%%%%%%%%%%%%%%%%%%%%%%%%%%%%%%%%%%%%%%%%%%%%%%%%%%%%%%%%%%%%%%%%%%%%%%%%%%%%%%%%
	\begin{theorem}\label{g85}\emph{(Minimum attained by an extended IVF)}. Let $\textbf{F}$ be $gH$-lsc, level-bounded and proper extended IVF. Then, the set  $\argmin_\mathcal{X}\textbf{F}$ is nonempty and compact.
	\end{theorem}
	 %%%%%%%%%%%%%%%%%%%%%%%%%%%%%%%%%%%%%%%%%%%%%%%%%%%%%%%%%%%%%%%%%%%%%%%%%%%%%%%%%%%%%%%%%%%%%%%%%%%%%%%%%%%%%%%%%%%%%%%%%%%%%%%%%%%%%%%%%%%%%%%%%%%%%%%%%%%%%%%
	\begin{proof}
		Let $\bar{\boldsymbol{\alpha}}=\inf\textbf{F}$.  So, $\bar{\boldsymbol{\alpha}}\prec [+\infty,+\infty]$ because $\textbf{F}$ is proper.\\
		
		Note that lev$_{\boldsymbol{\alpha}\nprec}\textbf{F}\neq
		\emptyset$ for any $\boldsymbol{\alpha}$ that satisfies $\bar{\boldsymbol{\alpha}}\prec\boldsymbol{\alpha}\prec [+\infty,+\infty]$. Also, as $\textbf{F}$ is level-bounded, lev$_{\boldsymbol{\alpha}\nprec}\textbf{F}$ is bounded and by Theorem \ref{n1}, it is also closed. Thus, lev$_{\boldsymbol{\alpha}\nprec}\textbf{F}$ is nonempty compact for $\bar{\boldsymbol{\alpha}}\prec\boldsymbol{\alpha}\prec [+\infty,+\infty]$ and are nested as lev$_{\boldsymbol{\alpha}\nprec}\textbf{F}\subseteq$ lev$_{\boldsymbol{\beta}\nprec}\textbf{F}$ when $\boldsymbol{\alpha}\prec\boldsymbol{\beta}.$ Therefore,\[\bigcap_{ \bar{\boldsymbol{\alpha}}\prec\boldsymbol{\alpha}\prec  +\infty}\text{lev}_{\boldsymbol{\alpha}\nprec}\textbf{F}=\text{lev}_{\bar{\boldsymbol{\alpha}}\nprec}\textbf{F}=\argmin_\mathcal{X}\textbf{F}\] is nonempty and compact.
	\end{proof}
	 %%%%%%%%%%%%%%%%%%%%%%%%%%%%%%%%%%%%%%%%%%%%%%%%%%%%%%%%%%%%%%%%%%%%%%%%%%%%%%%%%%%%%%%%%%%%%%%%%%%%%%%%%%%%%%%%%%%%%%%%%%%%%%%%%%%%%%%%%%%%%%%%%%%%%%%%%%%%%%%
	Next, we present a theorem which gives a characterization of the argument minimum set of an IVF in terms of $gH$-G\^{a}teaux differentiability. An IVF $\textbf{F}:\mathcal{X}\rightarrow I(\mathbb{R})$ is said to be $gH$-G\^{a}teaux differentiable (see \cite{ghosh2020generalized}) at $\bar{x}\in\mathcal{X}$ if  the limit
\[
\textbf{F}_\mathscr{G}(\bar{x})(h)=\lim_{\lambda \to 0+}\frac{1}{\lambda}\odot\left(\textbf{F}(\bar{x}+\lambda h)\ominus_{gH}\textbf{F}(\bar{x})\right)
\]
exists for all $h \in \mathcal{X}$ and $\textbf{F}_\mathscr{G}(\bar{x})$ is a $gH$-continuous linear IVF from $\mathcal{X}$ to $I(\mathbb{R})$. Then, we call $\textbf{F}_\mathscr{G}(\bar{x})$ as the $gH$-G\^{a}teaux derivative of $\textbf{F}$ at $\bar{x}$.
	
	 %%%%%%%%%%%%%%%%%%%%%%%%%%%%%%%%%%%%%%%%%%%%%%%%%%%%%%%%%%%%%%%%%%%%%%%%%%%%%%%%%%%%%%%%%%%%%%%%%%%%%%%%%%%%%%%%%%%%%%%%%%%%%%%%%%%%%%%%%%%%%%%%%%%%%%%%%%%%%%%
	\begin{theorem}\label{n2}\emph{(Characterization of the set argument minimum of an IVF)}.
		Let $\textbf{F}$ be an extended IVF and $\bar{x}\in\argmin_{x\in\mathcal{X}}\textbf{F}(x).$ If the function $\textbf{F}$ has a $gH$-G\^{a}teaux derivative at $\bar{x}$ in every direction $h\in\mathcal{X}$, then\[\textbf{F}_\mathscr{G}(\bar{x})(h)=\textbf{0}~\text{for all}~h\in\mathcal{X}.\]
	\end{theorem}
	 %%%%%%%%%%%%%%%%%%%%%%%%%%%%%%%%%%%%%%%%%%%%%%%%%%%%%%%%%%%%%%%%%%%%%%%%%%%%%%%%%%%%%%%%%%%%%%%%%%%%%%%%%%%%%%%%%%%%%%%%%%%%%%%%%%%%%%%%%%%%%%%%%%%%%%%%%%%%%%%
	\begin{proof}
		Observe that  any $\bar{x}\in\argmin_{x\in\mathcal{X}}\textbf{F}(x)$, is also an efficient point. Then, the proof follows from proof of the Theorem 4.2 in \cite{ghosh2020generalized}.
	\end{proof}
	 %%%%%%%%%%%%%%%%%%%%%%%%%%%%%%%%%%%%%%%%%%%%%%%%%%%%%%%%%%%%%%%%%%%%%%%%%%%%%%%%%%%%%%%%%%%%%%%%%%%%%%%%%%%%%%%%%%%%%%%%%%%%%%%%%%%%%%%%%%%%%%%%%%%%%%%%%%%%%%%%%%%%%%%%%%%%%%%%%%%%%%%%%%%%%%%%%%%%%%%%%%%%%%%%%%%%%%%%%%%%%%%%%%%%%%%%%%%%%%%%%%%%%%%%%%%%%%%%%%%%%%%%%%%%%%%%%%%%%%%%%%%%%%%%%%%%%%%%%%%%%%%%%%%%%%%%%%%%%%%%%%%%%%%%%%%%%%%%%%%%%%%%%%%%%%%%%%%%%%%%%%%%%%%%%%%%%%%%%%%%%%%%%%%%%%%%%%%%%%%%%%%%%%%%%%%%%%%%%%%%%%%%%%%%%%%%%%%%%%%%%%%%%%%%%%%%%%%%%%%%%%%%%%%%%%%%%%%%%%%%%%%%%%%%%%%%%%%%%%%%%%%%%%%%%%%%%%%%%%%%%%%%%%%%%%%%%%%%%%%%%%%%%%%%%%%%%%%%%%%%%%%%%%%%

	 %%%%%%%%%%%%%%%%%%%%%%%%%%%%%%%%%%%%%%%%%%%%%%%%%%%%%%%%%%%%%%%%%%%%%%%%%%%%%%%%%%%%%%%%%%%%%%%%%%%%%%%%%%%%%%%%%%%%%%%%%%%%%%%%%%%%%%%%%%%%%%%%%%%%%%%%%%%%%%%%%%%%%%%%%%%%%%%%%%%%%%%%%%%%%%%%%%%%%%%%%%%%%%%%%%%%%%%%%%%%%%%%%%%%%%%%%%%%%%%%%%%%%%%%%%%%%%%%%%%%%%%%%%%%%%%%%%%%%%%%%%%%%%%%%%%%%%%%%%%%%%%%%%%%%%%%%%%%%%%%%%%%%%%%%%%%%%%%%%%%%%%%%%%%%%%%%%%%%%%%%%%%%%%%%%%%%%%%%%%%%%%%%%%%%%%%%%%%%%%%%%%%%%%%%%%%%%
	\section{Ekeland's Variational Principle and its Applications}\label{sec5}
	In this section, we present the main results\textemdash Ekeland's variational principle for IVFs along with its application for $gH$-G\^{a}teaux differentiable IVFs.
	
	 %%%%%%%%%%%%%%%%%%%%%%%%%%%%%%%%%%%%%%%%%%%%%%%%%%%%%%%%%%%%%%%%%%%%%%%%%%%%%%%%%%%%%%%%%%%%%%%%%%%%%%%%%%%%%%%%%%%%%%%%%%%%%%%%%%%%%%%%%%%%%%%%%%%%%%%%%%%%%%%
	\begin{lemma}\label{g84}
		Let $\bar{x}\in \mathcal{X}$ and $\textbf{A}\in I(\mathbb{R})$. Then,  $\{x\in\mathcal{X}:\textbf{A}\nprec\lVert x-\bar{x}\rVert_{\mathcal{X}}\}$ is a bounded set.
	\end{lemma}
	 %%%%%%%%%%%%%%%%%%%%%%%%%%%%%%%%%%%%%%%%%%%%%%%%%%%%%%%%%%%%%%%%%%%%%%%%%%%%%%%%%%%%%%%%%%%%%%%%%%%%%%%%%%%%%%%%%%%%%%%%%%%%%%%%%%%%%%%%%%%%%%%%%%%%%%%%%%%%%%%
	\begin{proof}
		Let $\textbf{A}=[\underline{a},\overline{a}].$ Then,
		\begin{eqnarray*}
		&&\{x\in\mathcal{X}:\textbf{A}\nprec\lVert x-\bar{x}\rVert_{\mathcal{X}}\}\\&=&
			\{x\in\mathcal{X}:[\underline{a},\overline{a}]\nprec\lVert x-\bar{x}\rVert_{\mathcal{X}}\}\\&=&\{x\in\mathcal{X}:\lVert x-\bar{x}\rVert_{\mathcal{X}}\preceq [\underline{a},\overline{a}]~\text{or}~`[\underline{a},\overline{a}]~\text{and}~\Vert x-\bar{x}\rVert_{\mathcal{X}}~\text{are not comparable'}\}\\&=&\{x\in\mathcal{X}:`\lVert x-\bar{x} \rVert_{\mathcal{X}}\leq\underline{a}~\text{and}~\lVert x-\bar{x} \rVert_{\mathcal{X}}\leq\overline{a}\text{'}\\&&~~~~~~~~~~~~~~~~\text{or}~`[\underline{a},\overline{a}]~\text{and}~ \lVert x-\bar{x}\rVert_{\mathcal{X}}~\text{are not comparable'}\}\\&=&\{x\in\mathcal{X}:`\lVert x-\bar{x} \rVert_{\mathcal{X}}\leq\underline{a}\text{'}~\text{or}~`\lVert x-\bar{x} \rVert_{\mathcal{X}}<\underline{a}~\text{and}~\lVert x-\bar{x} \rVert_{\mathcal{X}}>\overline{a}\text{'}\\&&~~~~~~~~~~~~~~~~\text{or}~`\lVert x-\bar{x} \rVert_{\mathcal{X}}>\underline{a}~\text{and}~\lVert x-\bar{x} \rVert_{\mathcal{X}}<\overline{a}\text{'}\}\\&=&\{x\in\mathcal{X}:\lVert x-\bar{x} \rVert_{\mathcal{X}}\leq\underline{a}~\text{or}~\underline{a}<\lVert x-\bar{x} \rVert_{\mathcal{X}}<\overline{a}\},
		\end{eqnarray*}
		which is a bounded set.\\
		Hence, for any $\bar{x}\in\mathcal{X}$ and $\textbf{A}\in I(\mathbb{R})$, $\{x\in\mathcal{X}:\textbf{A}\nprec\lVert x-\bar{x}\rVert_{\mathcal{X}}\}$ is bounded.
	\end{proof}
	 %%%%%%%%%%%%%%%%%%%%%%%%%%%%%%%%%%%%%%%%%%%%%%%%%%%%%%%%%%%%%%%%%%%%%%%%%%%%%%%%%%%%%%%%%%%%%%%%%%%%%%%%%%%%%%%%%%%%%%%%%%%%%%%%%%%%%%%%%%%%%%%%%%%%%%%%%%%%%%%
	%\todo[inline, color=red!40]{This is our main result, so if needed please check its proof again}
	\begin{theorem}\label{g76}\emph{(Ekeland's variational principle for IVFs)}. Let $\textbf{F}:\mathcal{X}\rightarrow I(\mathbb{R})\cup \{+\infty\}$ be a $gH$-lsc extended IVF and $\epsilon>0$. Assume that \[\inf\limits_{\mathcal{X}}\textbf{F}~\text{is finite and}~\textbf{F}(\bar{x})\prec\inf\limits_{\mathcal{X}}\textbf{F}\oplus[\epsilon,\epsilon].\]
		Then, for any $\delta>0$, there exists an $x_0\in\mathcal{X}$ such that
		\begin{enumerate}[(i)]
			\item $\lVert x_0-\bar{x}\rVert_{\mathcal{X}}<\frac{\epsilon}{\delta}$,
			\item $\textbf{F}(x_0)\preceq\textbf{F}(\bar{x})$, and
			\item $\argmin\limits_{x\in\mathcal{X}}\{\textbf{F}(x)\oplus\delta\lVert x-x_0\rVert_{\mathcal{X}}\}=\{x_0\}.$
		\end{enumerate}
	\end{theorem}
	 %%%%%%%%%%%%%%%%%%%%%%%%%%%%%%%%%%%%%%%%%%%%%%%%%%%%%%%%%%%%%%%%%%%%%%%%%%%%%%%%%%%%%%%%%%%%%%%%%%%%%%%%%%%%%%%%%%%%%%%%%%%%%%%%%%%%%%%%%%%%%%%%%%%%%%%%%%%%%%%
	\begin{proof}
		Let $\bar{\boldsymbol{\alpha}}=\inf\limits_{\mathcal{X}}\textbf{F}$ and $\overline{\textbf{F}}(x)=\textbf{F}(x)\oplus\delta\lVert x-\bar{x}\rVert_{\mathcal{X}}.$ \\
		
		Since $\overline{\textbf{F}}$ is the sum of two $gH$-lsc and proper IVFs, $\overline{\textbf{F}}$ is $gH$-lsc by Theorem \ref{g83}. Also,
		\begin{eqnarray*}
		 \text{lev}_{\boldsymbol{\alpha}\nprec}\overline{\textbf{F}}&=&\left\{x\in\mathcal{X}:\boldsymbol{\alpha}\nprec\overline{\textbf{F}}(x)\right\}\\&=&\left\{x\in\mathcal{X}:\boldsymbol{\alpha}\nprec\textbf{F}(x)\oplus\delta\lVert x-\bar{x}\rVert_{\mathcal{X}}\right\}\\&\subseteq&\left\{x\in\mathcal{X}:\boldsymbol{\alpha}\nprec\bar{\boldsymbol{\alpha}}\oplus\delta\lVert x-\bar{x}\rVert_{\mathcal{X}}\right\}\\&=&\left\{x\in\mathcal{X}:\frac{\boldsymbol{\alpha}\ominus_{gH}\bar{\boldsymbol{\alpha}}}{\delta}\nprec\lVert x-\bar{x}\rVert_{\mathcal{X}}\right\}\\&=&\left\{x\in\mathcal{X}:\textbf{A}\nprec \lVert x-\bar{x}\rVert_{\mathcal{X}}\right\},~\text{where}~ \textbf{A}=\frac{\boldsymbol{\alpha}\ominus_{gH}\bar{\boldsymbol{\alpha}}}{\delta}.
		\end{eqnarray*}
		Therefore, by Lemma \ref{g84}, $\overline{\textbf{F}}$ is level-bounded. Clearly, $\overline{\textbf{F}}$ is proper. Hence, by Theorem \ref{g85}, $C=\argmin_{\mathcal{X}}\overline{\textbf{F}}$ is nonempty and compact. \\

		Let us consider the function $\tilde{\textbf{F}}=\textbf{F}\oplus\delta_{C}$ on $\mathcal{X}$. Note that $\tilde{\textbf{F}}$ is proper and level-bounded. Since $C$ is nonempty and compact, so by Remark \ref{g86}, $\delta_{C}$ is $gH$-lsc.
		Thus, by Theorem \ref{g83},  $\tilde{\textbf{F}}$ is $gH$-lsc, and hence by Theorem \ref{g85}, $\argmin_{\mathcal{X}}\tilde{\textbf{F}}$ is nonempty.  \\

		Let $x_0\in \argmin_{\mathcal{X}}\tilde{\textbf{F}}$. Then, over the set $C$, $\textbf{F}$ is minimum  at $x_0$. \\

 Since $x_0\in C$, $\overline{\textbf{F}}(x_0)\prec\overline{\textbf{F}}(x)$ for $x\notin C$. This implies that for any $x\notin C$,
\begin{eqnarray*}
&&\textbf{F}(x_0)\oplus\delta\lVert x_0-\bar{x}\rVert_{\mathcal{X}}\prec\textbf{F}(x)\oplus\delta\lVert x-\bar{x}\rVert_{\mathcal{X}}\\		
 &\implies& \textbf{F}(x_0)\prec\textbf{F}(x)\oplus\delta\lVert x-\bar{x}\rVert_{\mathcal{X}}\ominus_{gH}\delta\lVert x_0-\bar{x}\rVert_{\mathcal{X}}.
\end{eqnarray*}
Hence, $\textbf{F}(x_0)\prec\textbf{F}(x)\oplus\delta\lVert x-x_0\rVert_{\mathcal{X}}$ for all $x\notin C$ with $x\neq x_0$, and thus $\argmin_{x\in\mathcal{X}}\{\textbf{F}(x)\oplus\delta\lVert x-x_0\rVert_{\mathcal{X}}\}= \{x_0\}$.\\

		Also, as $x_0\in C$, we have $\overline{\textbf{F}}(x_0)\preceq\overline{\textbf{F}}(\bar{x})$, which implies
		\begin{eqnarray*}
		&&\overline{\textbf{F}}(x_0)\preceq\textbf{F}(\bar{x}) ~\text{because}~ \overline{\textbf{F}}(\bar{x})=\textbf{F}(\bar{x})\\
		&\implies&\textbf{F}(x_0)\oplus\delta\lVert  x_0-\bar{x}\rVert_{\mathcal{X}}\preceq\textbf{F}(\bar{x})\\&\implies&\textbf{F}(x_0)\preceq\textbf{F}(\bar{x})\ominus_{gH}\delta\lVert x_0-\bar{x}\rVert_{\mathcal{X}}\\&\implies&\textbf{F}(x_0)\prec\bar{\boldsymbol{\alpha}}\oplus[\epsilon,\epsilon]\ominus_{gH}\delta\lVert x_0-\bar{x}\rVert_{\mathcal{X}}~ \text{because}~ \textbf{F}(\bar{x})\prec\inf\limits_\mathcal{X}\textbf{F}\oplus[\epsilon,\epsilon]\\&\implies&\delta\lVert x_0-\bar{x}\rVert_{\mathcal{X}}\prec\bar{\boldsymbol{\alpha}}\oplus[\epsilon,\epsilon]\ominus_{gH}\textbf{F}(x_0)\\&\implies&\delta\lVert x_0-\bar{x}\rVert_{\mathcal{X}}\prec[\epsilon,\epsilon]~ \text{because}~ \bar{\boldsymbol{\alpha}}\ominus_{gH}\textbf{F}(x_0)\preceq \textbf{0}\\&\implies& \lVert x_0-\bar{x}\rVert_{\mathcal{X}}<\frac{\epsilon}{\delta}.
		\end{eqnarray*}
		This completes the proof.
	\end{proof}
	 %%%%%%%%%%%%%%%%%%%%%%%%%%%%%%%%%%%%%%%%%%%%%%%%%%%%%%%%%%%%%%%%%%%%%%%%%%%%%%%%%%%%%%%%%%%%%%%%%%%%%%%%%%%%%%%%%%%%%%%%%%%%%%%%%%%%%%%%%%%%%%%%%%%%%%%%%%%%%%%
	Next, we give an application of Ekeland's variational principle for IVFs. In order to do that we need the concept of norm of a bounded linear IVF. By a bounded linear IVF (see \cite{ghosh2020generalized}), we mean a linear IVF $\textbf{G}:\mathcal{X}\rightarrow I(\mathbb{R})$ for which  there exists a nonnegative real number $C$ such that
		\[\lVert \textbf{G}(x)\lVert_{I(\mathbb{R})}\leq C\lVert x\rVert_{\mathcal{X}} ~\text{for all}~ x\in\mathcal{X}.\]

	In the next lemma, we introduce norm for a bounded linear IVF.
	 %%%%%%%%%%%%%%%%%%%%%%%%%%%%%%%%%%%%%%%%%%%%%%%%%%%%%%%%%%%%%%%%%%%%%%%%%%%%%%%%%%%%%%%%%%%%%%%%%%%%%%%%%%%%%%%%%%%%%%%%%%%%%%%%%%%%%%%%%%%%%%%%%%%%%%%%%%%%%%%
	\begin{lemma}(\emph{Norm of a bounded linear IVF}). Let $\textbf{G}:\mathcal{X}\rightarrow I(\mathbb{R})$ be a bounded linear IVF. Then,
		\begin{equation*}\label{g87}
		\lVert\textbf{G}\rVert=\sup_{
			\substack{x\in\mathcal{X}\\ \lVert x\rVert_{\mathcal{X}}=1}}\lVert\textbf{G}(x)\rVert_{I(\mathbb{R})}
		\end{equation*}
		is a norm on the set of all bounded linear IVFs on $\mathcal{X}.$
	\end{lemma}
	 %%%%%%%%%%%%%%%%%%%%%%%%%%%%%%%%%%%%%%%%%%%%%%%%%%%%%%%%%%%%%%%%%%%%%%%%%%%%%%%%%%%%%%%%%%%%%%%%%%%%%%%%%%%%%%%%%%%%%%%%%%%%%%%%%%%%%%%%%%%%%%%%%%%%%%%%%%%%%%%
	\begin{proof}
		Observe that $\lVert\textbf{G}\rVert\geq 0$ for any bounded linear IVF $\textbf{G}$ and $\lVert\textbf{G}\rVert=0$ if and only if $\textbf{G}=\textbf{0}$. Let $\gamma\in\mathbb{R}$. We see that
		\begin{eqnarray*}
			&&\lVert\gamma\odot\textbf{G}\rVert\\&=&\sup_{
				\substack{x\in\mathcal{X}\\ \lVert x\rVert_{\mathcal{X}}=1}}\lVert (\gamma\odot\textbf{G})(x)\rVert_{I(\mathbb{R})}=\sup_{
				\substack{x\in\mathcal{X}\\ \lVert x\rVert_{\mathcal{X}}=1}}\lvert\gamma\rvert\lVert\textbf{G}(x)\rVert_{I(\mathbb{R})}\\&=&\lvert\gamma\rvert\sup_{
				\substack{x\in\mathcal{X}\\ \lVert x\rVert_{\mathcal{X}}=1}}\lVert\textbf{G}(x)\rVert_{I(\mathbb{R})}=\lvert\gamma\rvert\lVert\textbf{G}\rVert.
		\end{eqnarray*}
		Further,
		\begin{eqnarray*}
			\lVert\mathbf{G_{1}}\oplus\mathbf{G_2}\rVert&=&\sup_{
				\substack{x\in\mathcal{X}\\ \lVert x\rVert_{\mathcal{X}}=1}}\lVert(\mathbf{G_1}\oplus\mathbf{G_2})(x)\rVert_{I(\mathbb{R})}\\&=&\sup_{
				\substack{x\in\mathcal{X}\\ \lVert x\rVert_{\mathcal{X}}=1}}\lVert\mathbf{G_1}(x)\oplus\mathbf{G_2}(x)\rVert_{I(\mathbb{R})}\\&\leq&\sup_{
				\substack{x\in\mathcal{X}\\ \lVert x\rVert_{\mathcal{X}}=1}}(\lVert\mathbf{G_1}(x)\rVert_{I(\mathbb{R})}+\lVert\mathbf{G_2}(x)\rVert_{I(\mathbb{R})}),~\text{by (\ref{1}) of Lemma \ref{g23}}\\&=&\sup_{
				\substack{x\in\mathcal{X}\\ \lVert x\rVert_{\mathcal{X}}=1}}\lVert\mathbf{G_1}(x)\rVert_{I(\mathbb{R})}+\sup_{
				\substack{x\in\mathcal{X}\\ \lVert x\rVert_{\mathcal{X}}=1}}\lVert\mathbf{G_2}(x)\rVert_{I(\mathbb{R})}\\&=&\lVert\mathbf{G_1}\rVert+\lVert\mathbf{G_2}\rVert.
		\end{eqnarray*}
		Hence, the result follows.
	\end{proof}
	 %%%%%%%%%%%%%%%%%%%%%%%%%%%%%%%%%%%%%%%%%%%%%%%%%%%%%%%%%%%%%%%%%%%%%%%%%%%%%%%%%%%%%%%%%%%%%%%%%%%%%%%%%%%%%%%%%%%%%%%%%%%%%%%%%%%%%%%%%%%%%%%%%%%%%%%%%%%%%%%
	 %%%%%%%%%%%%%%%%%%%%%%%%%%%%%%%%%%%%%%%%%%%%%%%%%%%%%%%%%%%%%%%%%%%%%%%%%%%%%%%%%%%%%%%%%%%%%%%%%%%%%%%%%%%%%%%%%%%%%%%%%%%%%%%%%%%%%%%%%%%%%%%%%%%%%%%%%%%%%%%
	%\begin{lemma} (See \cite{ghosh2020generalized}).\label{g88}
% If a linear IVF $\mathbf{L}: \mathcal{X} \rightarrow I(\mathbb{R})$ is $gH$-continuous at the zero vector of $\mathcal{X}$, then it is a bounded linear IVF.
%	\end{lemma}
	 %%%%%%%%%%%%%%%%%%%%%%%%%%%%%%%%%%%%%%%%%%%%%%%%%%%%%%%%%%%%%%%%%%%%%%%%%%%%%%%%%%%%%%%%%%%%%%%%%%%%%%%%%%%%%%%%%%%%%%%%%%%%%%%%%%%%%%%%%%%%%%%%%%%%%%%%%%%%%%%
	\begin{theorem}
		Let $\textbf{G}:\mathcal{X}\rightarrow I(\mathbb{R})$ be a linear IVF. If $\textbf{G}$ is $gH$-continuous on $\mathcal{X}$, then $\textbf{G}$ is a bounded linear IVF.
	\end{theorem}
	 %%%%%%%%%%%%%%%%%%%%%%%%%%%%%%%%%%%%%%%%%%%%%%%%%%%%%%%%%%%%%%%%%%%%%%%%%%%%%%%%%%%%%%%%%%%%%%%%%%%%%%%%%%%%%%%%%%%%%%%%%%%%%%%%%%%%%%%%%%%%%%%%%%%%%%%%%%%%%%%
	\begin{proof}
		By the hypothesis, $\textbf{G}$ is $gH$-continuous at the zero vector of $\mathcal{X}$. Therefore, by Lemma 4.2 in \cite{ghosh2020generalized}, $\textbf{G}$ is a bounded linear IVF.
	\end{proof}
	 %%%%%%%%%%%%%%%%%%%%%%%%%%%%%%%%%%%%%%%%%%%%%%%%%%%%%%%%%%%%%%%%%%%%%%%%%%%%%%%%%%%%%%%%%%%%%%%%%%%%%%%%%%%%%%%%%%%%%%%%%%%%%%%%%%%%%%%%%%%%%%%%%%%%%%%%%%%%%%%
	As an application of Theorem \ref{g76}, we give a variational principle for $gH$-G\^{a}teaux differentiable IVFs.
	 %%%%%%%%%%%%%%%%%%%%%%%%%%%%%%%%%%%%%%%%%%%%%%%%%%%%%%%%%%%%%%%%%%%%%%%%%%%%%%%%%%%%%%%%%%%%%%%%%%%%%%%%%%%%%%%%%%%%%%%%%%%%%%%%%%%%%%%%%%%%%%%%%%%%%%%%%%%%%%%
\begin{theorem}\label{g96}\emph{(Variational principle for $gH$-G\^{a}teaux differentiable IVFs)}. Let $\textbf{F}:\mathcal{X}\rightarrow I(\mathbb{R})\cup \{+\infty\}$ be a $gH$-lsc and $gH$-G\^{a}teaux differentiable extended IVF, and $\epsilon>0$. Suppose that \[\inf\limits_{\mathcal{X}}\textbf{F}~ \text{is finite and}~ \textbf{F}(\bar{x})\prec\inf\limits_\mathcal{X}\textbf{F}\oplus[\epsilon,\epsilon].\]
		Then, for any $\delta>0$, there exists an $x_0\in\mathcal{X}$ such that
		\begin{enumerate}[(i)]
			\item\label{g89} $\lVert x_0-\bar{x}\rVert_{\mathcal{X}}<\frac{\epsilon}{\delta}$,
			\item \label{g90}$\textbf{F}(x_0)\preceq\textbf{F}(\bar{x})$, and
			\item $\lVert\textbf{F}_\mathscr{G}(x_0)\rVert\leq\delta.$
		\end{enumerate}
	\end{theorem}
	 %%%%%%%%%%%%%%%%%%%%%%%%%%%%%%%%%%%%%%%%%%%%%%%%%%%%%%%%%%%%%%%%%%%%%%%%%%%%%%%%%%%%%%%%%%%%%%%%%%%%%%%%%%%%%%%%%%%%%%%%%%%%%%%%%%%%%%%%%%%%%%%%%%%%%%%%%%%%%%%
	\begin{proof}
		By Theorem \ref{g76}, there exists an $x_0\in\mathcal{X}$ that satisfies (\ref{g89}) and (\ref{g90}), and $x_0\in \argmin_{x\in\mathcal{X}}\{\textbf{F}(x)\oplus\delta\lVert x-x_0\rVert_{\mathcal{X}}\}$. Therefore, $\textbf{F}(x_0)\preceq\textbf{F}(x)\oplus\delta\lVert x-x_0\rVert_{\mathcal{X}}$ and hence
		\begin{equation}\label{g91}
		\textbf{F}(x_0)\ominus_{gH}\delta\lVert x-x_0\rVert_{\mathcal{X}}\preceq\textbf{F}(x).
		\end{equation}
		Take any $h\in\mathcal{X}$ and set $x=x_0+th$ in the equation (\ref{g91}) with $t>0$. Then, we get\[\textbf{F}(x_0)\ominus_{gH}\delta\lVert th\rVert_{\mathcal{X}}\preceq\textbf{F}(x_0+th).\]
		Thus, \[-\delta\lVert h\rVert_{\mathcal{X}}\preceq\frac{1}{t}\odot\left(\textbf{F}(x_0+t h)\ominus_{gH}\textbf{F}(x_0)\right).\]
		Letting $t\rightarrow 0+$, we get
		\[-\delta\lVert h\rVert_{\mathcal{X}}\preceq\textbf{F}_\mathscr{G}(x_0)(h).\]
		Taking the infimum on both sides over all $h\in \mathcal{X}$ with $\lVert h\rVert_{\mathcal{X}}=1$, we get
		\[-\delta \leq -\lVert\textbf{F}_\mathscr{G}(x_0)\rVert,~\text{or},~ \lVert\textbf{F}_\mathscr{G}(x_0)\rVert\leq \delta.\]
		This completes the proof.
	\end{proof}
	 %%%%%%%%%%%%%%%%%%%%%%%%%%%%%%%%%%%%%%%%%%%%%%%%%%%%%%%%%%%%%%%%%%%%%%%%%%%%%%%%%%%%%%%%%%%%%%%%%%%%%%%%%%%%%%%%%%%%%%%%%%%%%%%%%%%%%%%%%%%%%%%%%%%%%%%%%%%%%%%
	The importance of the Theorem \ref{g96} is that in the absence of points belonging to the set $\argmin_{x\in\mathcal{X}}\textbf{F}(x)$, we can capture a point $x_0$ that almost minimizes $\textbf{F}$. In other words, the equations $\textbf{F}(x_0)=\inf\limits_{\mathcal{X}}\textbf{F}$ and $\textbf{F}_\mathscr{G}(x_0)=\textbf{0}$ can be satisfied to any prescribed accuracy $\delta>0.$\\
	
%%%%%%%%%%%%%%%%%%%%%%%%%%%%%%%%%%%%%%%%%%%%%%%%%%%%%%%%%%%%%%%%%%%%%%%%%%%%%%%%%%%%%%%%%%%%%%%%%%%%%%%%%%%%%%%%%%%%%%%%%%%%%%%%%%%%%%%%%%%%%%%%%%%%%%%%%%%%%%%%%%%%%%%%%%%%%%%%%%%%%%%%%%%%%%%%%%%%%%%%%%%%%%%%%%%%%%%%%%%%%%%%%%%%%%%%%%%%%%%%%%%%%%%%%%%%%%%%%%%%%%%%%%%%%%%%%%%%%%%%%%%%%%%%%%%%%%%%%%%%%%%%%%%%%%%%%%%%%%%%%%%%%%%%%%%%%%%%%%%%
%%%%%%%%%%%%%%%%%%%%%%%%%%%%%%%%%%%%%%%%%%%%%%%%%%%%%%%%%%%%%%%%%%%%%%%%%%%%%%%%%%%%%%%%%%%%%%%%%%%%%%%%%%%%%%%%%%%%%%%%%%%%%%%%%%%%%%%%%%%%%%%%%%%%%%%%%%%%%%%
\section{Discussion and Conclusion}\label{sec6}
In this article, the concept of $gH$-semicontinuity (Definitions \ref{g5} and \ref{g8}) has been introduced for IVFs. Their interrelation with $gH$-continuity has been shown (Theorem \ref{g44}). The concept of sequence of intervals is used to give a characterization of lower and upper limits of extended IVFs (Lemmas \ref{g20} and \ref{g50}). By using a characterization of $gH$-lower semicontinuity for IVFs (Theorem \ref{n1}), it has been reported that an extended $gH$-lsc, level-bounded and proper IVF always attains its minimum (Theorem \ref{g85}). A characterization of the set of argument minimum of an IVF has been provided with the help of $gH$-G\^{a}teaux differentiability (Theorem \ref{n2}). We have further presented Ekeland's variational principle for IVFs (Theorem \ref{g76}). The proposed Ekeland's variational principle has been applied to find variational principle for $gH$-G\^{a}teaux differentiable IVFs (Theorem \ref{g96}).\\

In this article, we have considered analyzing closed and bounded intervals and IVFs whose values are closed and bounded intervals. A future study can be performed for other types of intervals. The analysis for other types of intervals is important because if we do not restrict the study for closed and bounded intervals the supremum of a set of closed and bounded intervals may become an open interval. For instance, for $\textbf{S}=\left\{\left[1-\frac{1}{n},2-\frac{1}{n}\right]:n\in\mathbb{N}\right\}$, $\sup\textbf{S}=(1,2).$\\

Immediately in the next step, we shall consider to solve the following two problems as the applications of the proposed study.
\begin{enumerate}[\textbf{Problem} 1.]
    \item The applications of the proposed variational principles in control systems in imprecise or uncertain environment will be shown shortly. Study of a control system in imprecise environment eventually appears due to the incomplete information (e.g., demand for a product) or unpredictable changes (e.g., changes in the climate) in the system. The general control problem in an imprecise or uncertain environment that we shall consider to study is the following:
    \begin{align*}
        \min~&~\textbf{G}(x(T))\\
        \text{subject to}~&~\frac{dx}{dt} = \textbf{F}(t, x(t), u(t)), \\
        ~&~ x(0) = x_0\in C_0,~x(T)\in C_1,
    \end{align*}
    where $C_0$ and $C_1$ are closed subsets of ${\mathbb{R}}^n$; $x:[0,T]\rightarrow {\mathbb{R}}^n$ and $u:[0,T]\rightarrow K$ are state and control variables, respectively, for some metrizable subset $K$ of ${\mathbb{R}}^n$; $\textbf{F}:[0,T]\times {\mathbb{R}}^n\times K\rightarrow I(\mathbb{R})$ is a $gH$-continuous IVF and $\textbf{G}:{\mathbb{R}}^n\rightarrow I(\mathbb{R})$ is a $gH$-Fr\'{e}chet differentiable IVF. To solve this system, the procedure adopted by Clarke in \cite{clarke1976maximum} may be useful.
    \item  We shall attempt to give optimality conditions for the following IOP, where $\mathcal{X}$ and $\mathcal{Y}$ are finite dimensional Banach spaces, $C$ is a nonempty closed subset of $\mathcal{X}\times \mathcal{Y}$, and $S$ is a closed convex subset of $\mathcal{Y}$:
\begin{align*}
    \min~&~ \textbf{F}(x,y)\\
    \text{subject to}~&~\mathbf{g}_{i}(x, y)\preceq \textbf{0},~i= 1,2,\cdots, m,\\
    ~&~\mathbf{h}_{j}(x, y)= \textbf{0},~j= 1,2,\cdots, k,\\
    ~&~(x, y)\in C,\\
    ~&~	y\in S,~ \big\langle F(x, y), y-z\big\rangle \leq 0\,\text{ for
			all} ~z\in S,
\end{align*}
	where $\textbf{F}: \mathcal{X}\times \mathcal{Y}\rightarrow I({\mathbb{R}}), ~\mathbf{g}_{i}: \mathcal{X}\times \mathcal{Y}\rightarrow  I(\mathbb{R})\cup \{+\infty\},~ i= 1,2,\cdots, ~m,~ \mathbf{h}_{j}: \mathcal{X}\times \mathcal{Y}\rightarrow I(\mathbb{R})\cup \{+\infty\}, ~j= 1, 2, \cdots,~ k,~ F: \mathcal{X}\times \mathcal{Y}\rightarrow \mathcal{Y},~\text{and}~\big\langle F(x, y), y-z\big\rangle$ denotes an inner product of $F(x, y)$ and $y-z.$
\end{enumerate}
Also, with the help of the proposed Ekeland's variational principle, in future, we shall try to investigate the concept of weak sharp minima \cite{burke2002weak} for IVFs and use it for sensitivity analysis of IOPs. \\

In parallel to the research proposed on IVFs, the research of fuzzy-valued functions (FVFs) may be another interesting path for future study. We hope that some FVF results would be similarly obtained to this article.
	 %%%%%%%%%%%%%%%%%%%%%%%%%%%%%%%%%%%%%%%%%%%%%%%%%%%%%%%%%%%%%%%%%%%%%%%%%%%%%%%%%%%%%%%%%%%%%%%%%%%%%%%%%%%%%%%%%%%%%%%%%%%%%%%%%%%%%%%%%%%%%%%%%%%%%%%%%%%%%%%%%%%%%%%%%%%%%%%%%%%%%%%%%%%%%%%%%%%%%%%%%%%%%%%%%%%%%%%%%%%%%%%%%%%%%%%%%%%%%%%%%%%%%%%%%%%%%%%%%%%%%%%%%%%%%%%%%%%%%%%%%%%%%%%%%%%%%%%%%%%%%%%%%%%%%%%%%%%%%%%%%%%%%%%%%%%%%%%%%%%%%%%%%%%%%%%%%%%%%%%%%%%%%%%%%%%%%%%%%%%%%%%%%%%%%%%%%%%%%%%%%%%%%%%%%%%%%%%%%%%%%%%%%%%%%%%%%%%%%%%%%%%%%%%%%%%%%%%%%%%%%%%%%%%%%%%%%%%%%%%%%%%%%%%%%
	\appendix
	\section{Proof of the Lemma \ref{g23}}\label{g24}
	\begin{proof}
	$\\$
		\emph{Proof of} (\ref{1}). Let $\textbf{A}=[\underline{a},\overline{a}]$ and $\textbf{B}=[\underline{b},\overline{b}].$ Then,
		$$\lVert\textbf{A}\oplus \textbf{B}\rVert_{I(\mathbb{R})}=\lVert[\underline{a},\overline{a}]\oplus [\underline{b},\overline{b}]\rVert_{I(\mathbb{R})}=\lVert [\underline{a}+\underline{b},\overline{a}+\overline{b}]\rVert_{I(\mathbb{R})}=\max \{\lvert \underline{a}+\underline{b}\rvert,\lvert\overline{a}+\overline{b}\rvert\}.$$
		We now have the following two possible cases.
		\begin{enumerate}[$\bullet$ \textbf{Case} 1.]
			\item $\lVert\textbf{A}\oplus \textbf{B}\rVert_{I(\mathbb{R})}=\lvert \underline{a}+\underline{b}\rvert.$\\
			Since $\lvert \underline{a}+\underline{b}\rvert\leq \lvert \underline{a}\rvert+\lvert \underline{b}\rvert\leq \max\{\lvert\underline{a}\rvert,\lvert\overline{a}\rvert\}+\max\{\lvert\underline{b}\rvert,\lvert\overline{b}\rvert\}=\lVert\textbf{A}\rVert_{I(\mathbb{R})}+\lVert\textbf{B}\rVert_{I(\mathbb{R})},$\\
			we get $\lVert\textbf{A}\oplus \textbf{B}\rVert_{I(\mathbb{R})}\leq\lVert\textbf{A}\rVert_{I(\mathbb{R})}+\lVert\textbf{B}\rVert_{I(\mathbb{R})}.$
			\item $\lVert\textbf{A}\oplus \textbf{B}\rVert_{I(\mathbb{R})}=\lvert \overline{a}+\overline{b}\rvert.$\\
			Since $\lvert \overline{a}+\overline{b}\rvert\leq \lvert \overline{a}\rvert+\lvert \overline{b}\rvert\leq\max\{\lvert\underline{a}\rvert,\lvert\overline{a}\rvert\}+\max\{\lvert\underline{b}\rvert,\lvert\overline{b}\rvert\}=\lVert\textbf{A}\rVert_{I(\mathbb{R})}+\lVert\textbf{B}\rVert_{I(\mathbb{R})},$\\
			therefore,  $\lVert\textbf{A}\oplus \textbf{B}\rVert_{I(\mathbb{R})}\leq\lVert\textbf{A}\rVert_{I(\mathbb{R})}+\lVert\textbf{B}\rVert_{I(\mathbb{R})}.$
		\end{enumerate}
		Hence, $\lVert \textbf{A}\oplus \textbf{B}\rVert_{I(\mathbb{R})}\leq \lVert\textbf{A}\rVert_{I(\mathbb{R})}+\lVert\textbf{B}\rVert_{I(\mathbb{R})}~\text{for all}~\textbf{A},~\textbf{B}\in I(\mathbb{R}).$\\ \\
		\emph{Proof of} (\ref{2}). Let $\textbf{A}=[\underline{a},\overline{a}],~\textbf{B}=[\underline{b},\overline{b}],~\textbf{C}=[\underline{c},\overline{c}]$ and $\textbf{D}=[\underline{d},\overline{d}]$.\\
		We note that
		\begin{equation}\label{g3}
		\textbf{A}\preceq \textbf{C}\implies [\underline{a},\overline{a}]\preceq [\underline{c},\overline{c}]\implies \underline{a}\leq \underline{c}~\text{and}~ \overline{a}\leq \overline{c}.
		\end{equation}\\
		Also,
		\begin{equation}\label{g4}
		\textbf{B}\preceq \textbf{D}\implies [\underline{b},\overline{b}]\preceq [\underline{d},\overline{d}]\implies \underline{b}\leq \underline{d} ~\text{and}~\overline{b}\leq \overline{d}.
		\end{equation}
		From (\ref{g3}) and (\ref{g4}), we have
		\begin{eqnarray*}
		&&\underline{a}+\underline{b}\leq \underline{c}+\underline{d} ~\text{and}~ \overline{a}+\overline{b}\leq \overline{c}+\overline{d}\\
	 &\implies&[\underline{a}+\underline{b},\overline{a}+\overline{b}]\preceq[\underline{c}+\underline{d},\overline{c}+\overline{d}].
		\end{eqnarray*}
		Thus, $\textbf{A}\oplus \textbf{B}\preceq \textbf{C}\oplus \textbf{D}$.\\ \\
	\end{proof}
	 %%%%%%%%%%%%%%%%%%%%%%%%%%%%%%%%%%%%%%%%%%%%%%%%%%%%%%%%%%%%%%%%%%%%%%%%%%%%%%%%%%%%%%%%%%%%%%%%%%%%%%%%%%%%%%%%%%%%%%%%%%%%%%%%%%%%%%%%%%%%%%%%%%%%%%%%%%%%%%%
	\section{Proof of the Lemma \ref{g28}}\label{g35}
	\begin{proof}
		$\\$
	
		\emph{Proof of} (\ref{6}). Let $\textbf{A}=[\underline{a},\overline{a}],~\textbf{B}=[\underline{b},\overline{b}]$ and $\epsilon>0.$\\
		$\textbf{A}\ominus_{gH}\textbf{B}=[\underline{a}-\underline{b},\overline{a}-\overline{b}]$ or $[\overline{a}-\overline{b}, \underline{a}-\underline{b}].$ Let us now consider the following four possible cases.
		\begin{enumerate}[$\bullet$ \textbf{Case} 1.]
			\item\label{g108} $\textbf{A}\ominus_{gH}\textbf{B}=[\underline{a}-\underline{b},\overline{a}-\overline{b}]$ and $\lVert \textbf{A}\ominus_{gH} \textbf{B}\rVert_{I(\mathbb{R})}=\lvert\underline{a}-\underline{b}\rvert.$\\
			So, we have
			\begin{equation}\label{g36}
			\underline{a}-\underline{b}\leq \overline{a}-\overline{b}~\text{and}~\lvert\overline{a}-\overline{b}\rvert\leq\lvert\underline{a}-\underline{b}\rvert.
			\end{equation}
			Let $\lVert \textbf{A}\ominus_{gH} \textbf{B}\rVert_{I(\mathbb{R})}<\epsilon$. Then,
			\begin{equation}\label{g37}
			\lvert\underline{a}-\underline{b}\rvert<\epsilon.
			\end{equation}
			By equation (\ref{g37}), we have $-\epsilon<\underline{a}-\underline{b}<\epsilon$, and hence $\underline{b}-\epsilon<\underline{a}.$\\
		 By equations (\ref{g36}) and (\ref{g37}), we have $			\lvert\overline{a}-\overline{b}\rvert<\epsilon$. This implies $\overline{b}-\epsilon<\overline{a}.$
		%	\begin{eqnarray*}
		%	&&\lvert\overline{a}-\overline{b}\rvert<\epsilon\\&\implies& -\epsilon<\overline{a}-\overline{b}<\epsilon\\&\implies&\overline{b}-\epsilon<\overline{a}.
		%	\end{eqnarray*}
			Therefore, $\textbf{B}\ominus_{gH}[\epsilon,\epsilon]=[\underline{b}-\epsilon,\overline{b}-\epsilon]\prec[\underline{a},\overline{a}]=\textbf{A}.$\\
			Note that  by equation (\ref{g37}),  $\underline{a}<\underline{b}+\epsilon$. Also, by equations (\ref{g36}) and (\ref{g37}), we have $\lvert\overline{a}-\overline{b}\rvert<\epsilon$. This implies $\overline{a}<\overline{b}+\epsilon.$
			%\begin{eqnarray*}
			%&&\lvert\overline{a}-\overline{b}\rvert<\epsilon\\&\implies& -\epsilon<\overline{a}-\overline{b}<\epsilon\\&\implies& \overline{a}<\overline{b}+\epsilon.
			%\end{eqnarray*}
			Therefore, $\textbf{A}=[\underline{a},\overline{a}]\prec [\underline{b}+\epsilon,\overline{b}+\epsilon]=\textbf{B}\oplus[\epsilon,\epsilon].$
			\item\label{g109} $\textbf{A}\ominus_{gH}\textbf{B}=[\underline{a}-\underline{b},\overline{a}-\overline{b}]$ and $\lVert \textbf{A}\ominus_{gH} \textbf{B}\rVert_{I(\mathbb{R})}=\lvert\overline{a}-\overline{b}\rvert.$\\
			So, we have
			\begin{equation}\label{g38}
			\underline{a}-\underline{b}\leq \overline{a}-\overline{b}~\text{and}~\lvert\underline{a}-\underline{b}\rvert\leq\lvert\overline{a}-\overline{b}\rvert.
			\end{equation}
			Consider
			\begin{eqnarray}\label{g39}
			&&\lVert \textbf{A}\ominus_{gH} \textbf{B}\rVert_{I(\mathbb{R})}<\epsilon\nonumber\\&\implies&\lvert\overline{a}-\overline{b}\rvert<\epsilon
			\end{eqnarray}
			By equation (\ref{g39}), we have
			\begin{equation*}
			\overline{b}-\epsilon<\overline{a}.
			\end{equation*}
			By equations (\ref{g38}) and \ref{g39}), we have  $\lvert\underline{a}-\underline{b}\rvert<\epsilon$. This implies $\underline{b}-\epsilon<\underline{a}.$
			%\begin{eqnarray*}
		%	&&\lvert\underline{a}-\underline{b}\rvert<\epsilon\\&\implies& -\epsilon<\underline{a}-\underline{b}<\epsilon\\&\implies&\underline{b}-\epsilon<\underline{a}.
		%	\end{eqnarray*}
			Therefore, $\textbf{B}\ominus_{gH}[\epsilon,\epsilon]=[\underline{b}-\epsilon,\overline{b}-\epsilon]\prec[\underline{a},\overline{a}].$\\
			Note that  by equation (\ref{g39}), $\overline{a}<\overline{b}+\epsilon$. Also, by equations (\ref{g38}) and (\ref{g39}), we have $\lvert\underline{a}-\underline{b}\rvert<\epsilon$. This implies $\underline{a}<\underline{b}+\epsilon.$
		%	\begin{eqnarray*}
		%	&&\lvert\underline{a}-\underline{b}\rvert<\epsilon
		%	\\&\implies& %-\epsilon<\underline{a}-\underline{b}<\epsilon\\&\implies&\unde%rline{a}<\underline{b}+\epsilon.
		%	\end{eqnarray*}
			Therefore, $\textbf{A}=[\underline{a},\overline{a}]\prec[\underline{b}+\epsilon,\overline{b}+\epsilon]=\textbf{B}\oplus[\epsilon,\epsilon].$
			\item $\textbf{A}\ominus_{gH}\textbf{B}=[\overline{a}-\overline{b},\underline{a}-\underline{b}]$ and $\lVert\textbf{A}\ominus_{gH}\textbf{B}\rVert_{I(\mathbb{R})}=\lvert\underline{a}-\underline{b}\rvert.$\\
			This case can be proved by following the steps similar to \textbf{Case} \ref{g108}.
			\item $\textbf{A}\ominus_{gH}\textbf{B}=[\overline{a}-\overline{b},\underline{a}-\underline{b}]$ and $\lVert\textbf{A}\ominus_{gH}\textbf{B}\rVert_{I(\mathbb{R})}=\lvert\overline{a}-\overline{b}\rvert.$\\
			This case can be proved by following the steps similar to \textbf{Case} \ref{g109}.\\
		\end{enumerate}
			Conversely, let $\textbf{B}\ominus_{gH}[\epsilon,\epsilon]\prec \textbf{A}\prec \textbf{B}\oplus[\epsilon,\epsilon].$\\
			Note that
			\begin{eqnarray}\label{g40}
			 \textbf{B}\ominus_{gH}[\epsilon,\epsilon]\prec\textbf{A}&\implies&[\underline{b}-\epsilon,\overline{b}-\epsilon]\prec[\underline{a},\overline{a}]\nonumber\\&\implies&\underline{b}-\epsilon<\underline{a}~\text{and}~\overline{b}-\epsilon<\overline{a}.
			\end{eqnarray}
			Also,
			\begin{eqnarray}\label{g41}
			 \textbf{A}\prec\textbf{B}\oplus[\epsilon,\epsilon]&\implies&[\underline{a},\overline{a}]\prec[\underline{b}+\epsilon,\overline{b}+\epsilon]\nonumber\\&\implies&\underline{a}<\underline{b}+\epsilon~\text{and}~\overline{a}<\overline{b}+\epsilon.
			\end{eqnarray}
			From equations (\ref{g40}) and (\ref{g41}), we have
			\begin{eqnarray*}
				 &&\underline{b}-\epsilon<\underline{a}<\underline{b}+\epsilon~\text{and}~\overline{b}-\epsilon<\overline{a}<\overline{b}+\epsilon\\&\implies&\lvert\underline{a}-\underline{b}\rvert<\epsilon~\text{and}~\lvert\overline{a}-\overline{b}\rvert<\epsilon\\&\implies&\max\{\lvert\underline{a}-\underline{b}\rvert,\lvert\overline{a}-\overline{b}\rvert\}<\epsilon\\&\text{ i.e.,}&\lVert\textbf{A}\ominus_{gH}\textbf{B}\rVert_{I(\mathbb{R})}<\epsilon.
			\end{eqnarray*}
			This completes the proof of (\ref{6}).\\ \\
		\noindent\emph{Proof of} (\ref{8}). Let $\textbf{A}=[\underline{a},\overline{a}],~\textbf{B}=[\underline{b},\overline{b}]~\text{and}~\epsilon>0.$\\
		Consider $\textbf{A}\ominus_{gH}[\epsilon,\epsilon]\nprec\textbf{B}$. This implies $[\underline{a}-\epsilon,\overline{a}-\epsilon]\nprec[\underline{b},\overline{b}].$ Thus, $`\underline{b}\leq\underline{a}-\epsilon~\text{ and}~\overline{b}\leq\overline{a}-\epsilon$' or $`\underline{b}<\underline{a}-\epsilon~\text{and}~\overline{b}>\overline{a}-\epsilon$' or $`\underline{b}>\underline{a}-\epsilon~\text{and}~\overline{b}<\overline{a}-\epsilon$'.
		Let us consider all these three possibilities in the following three cases.
		\begin{enumerate}[$\bullet$ \textbf{Case} 1.]
			\item $\underline{b}\leq\underline{a}-\epsilon~\text{and}~\overline{b}\leq\overline{a}-\epsilon$.\\
			So, we have
			\begin{eqnarray*}
			 &&\underline{a}>\underline{b}~\text{and}~\overline{a}>\overline{b},~\text{because}~\epsilon>0\\&\implies&\textbf{B}\prec\textbf{A}\implies\textbf{A}\npreceq\textbf{B}.
			\end{eqnarray*}
			\item $\underline{b}<\underline{a}-\epsilon~\text{and}~\overline{b}>\overline{a}-\epsilon$.\\
			Since $\underline{b}<\underline{a}-\epsilon$, so $\underline{a}>\underline{b}$, and thus $\textbf{A}\npreceq\textbf{B}.$
			\item $\underline{b}>\underline{a}-\epsilon~\text{and}~\overline{b}<\overline{a}-\epsilon.$\\
			Since $\overline{b}<\overline{a}-\epsilon$, so $\overline{a}>\overline{b},~\text{and thus}~\textbf{A}\npreceq\textbf{B}.$
		\end{enumerate}
			Hence, proof of (\ref{8}) is complete.
	\end{proof}
	 %%%%%%%%%%%%%%%%%%%%%%%%%%%%%%%%%%%%%%%%%%%%%%%%%%%%%%%%%%%%%%%%%%%%%%%%%%%%%%%%%%%%%%%%%%%%%%%%%%%%%%%%%%%%%%%%%%%%%%%%%%%%%%%%%%%%%%%%%%%%%%%%%%%%%%%%%%%%%%%
	\noindent
	\\ \\
	
	\noindent\textbf{Acknowledgement}\\ \\
	The first author is grateful to the Department of Science and Technology, India, for the award of `inspire fellowship' (DST/INSPIRE Fellowship/2017/IF170248).  Authors extend sincere thanks to Prof. Jos{\'e} Luis Verdegay, Universidad de Granada, for his valuable comments to improve quality of the paper.
	 %%%%%%%%%%%%%%%%%%%%%%%%%%%%%%%%%%%%%%%%%%%%%%%%%%%%%%%%%%%%%%%%%%%%%%%%%%%%%%%%%%%%%%%%%%%%%%%%%%%%%%%%%%%%%%%%%%%%%%%%%%%%%%%%%%%%%%%%%%%%%%%%%%%%%%%%%%%%%%%%%%%%%%%%%%%%%%%%%%%%%%%%%%%%%%%%%%%%%%%%%%%%%%%%%%%%%%%%%%%%%%%%%%%%%%%%%%%%%%%%%%%%%%%%%%%%%%%%%%%%%%%%%%%%%%%%%%%%%%%%%%%%%%%%%%%%%%%%%%%%%%%%%%%%%%%%%%%%%%%%%%%%%%%%%%%%%%%%%%%%%%%%%%%%%%%%%%%%%%%%%%%%%%%%%%%%%%%%%%%%%%%%%%%%%%%%%%%%%%%%%%%%%%%%%%%%%%%%%%%%%%%%%%%%%%%%%%%%%%%%%%%%%%%%%%%%%%%%%%%%%%%%%%%%%%%%%%%%
	%\newpage
	%\section*{References}
	\bibliographystyle{elsarticle-harv}
	\bibliography{mybib}
	 %%%%%%%%%%%%%%%%%%%%%%%%%%%%%%%%%%%%%%%%%%%%%%%%%%%%%%%%%%%%%%%%%%%%%%%%%%%%%%%%%%%%%%%%%%%%%%%%%%%%%%%%%%%%%%%%%%%%%%%%%%%%%%%%%%%%%%%%%%%%%%%%%%%%%%%%%%%%%%%%%%%%%%%%%%%%%%%%%%%%%%%%%%%%%%%%%%%%%%%%%%%%%%%%%%%%%%%%%%%%%%%%%%%%%%%%%%%%%%%%%%%%%%%%%%%%%%%%%%%%%%%%%%%%%%%%%%%%%%%%%%%%%%%%%%%%%%%%%%%%%%%%%%%%%%%%%%%%%%%%%%%%%%%%%%%%%%%%%%%%%%%%%%%%%%%%%%%%%%%%%%%%%%%%%%%%%%%%%%%%%%%%%%%%%%%%%%%%%%%%%%%%%%%%%%%%%%
	
\end{document}